\documentclass[11pt]{article}
\usepackage{amsmath, amssymb, amscd, amsthm, amsfonts}
\usepackage{graphicx}
\usepackage{hyperref}
\usepackage{float}
\usepackage{comment}

\title{Labeled Chip-Firing on Undirected $k$-ary Trees}
\author{Ryota Inagaki \and Aaron Lin}
\date{}

\usepackage[utf8]{inputenc}
\usepackage{float}
\newtheorem{theorem}{Theorem}[section]
\newtheorem{corollary}[theorem]{Corollary}
\newtheorem{lemma}[theorem]{Lemma}
\newtheorem{prop}[theorem]{Proposition}

\newtheorem{conjecture}[theorem]{Conjecture}

\theoremstyle{remark}
\newtheorem*{remark*}{Remark}

\theoremstyle{definition}
\newtheorem{definition}[theorem]{Definition}
\newtheorem{example}[theorem]{Example}

\allowdisplaybreaks

\begin{document}

\maketitle

\begin{abstract}
We explore labeled chip-firing on undirected $k$-ary trees, trees where every vertex has degree $k+1$. First, we extend known results for binary trees from Musiker and Nguyen, including the endgame and the locations of the smallest and largest chips, as well as relations between chips at different vertices. Then, inspired by recent work on the binary tree by the first author, Khovanova, and Luo, we use these properties to construct an upper bound, which we call the zigzag bound, on the number of stable configurations in labeled chip-firing on $k$-ary trees with $\frac{k^{\ell}-1}{k-1}$ labeled chips starting at the root. We further provide a novel lower bound on the number of stable configurations of $k$-ary trees, complementing our upper bounds.
\end{abstract}

\renewcommand{\thefootnote}{\fnsymbol{footnote}} 

\footnotetext{\emph{2020 Mathematics Subject Classification}:  05C57, 05C63, 05A15}

\footnotetext{\emph{Key words and phrases: } Labeled Chip-Firing, Undirected Trees, Vertex-firings} 

\section{Introduction}

Chip-firing is a combinatorial process that acts as a discrete approximation of a dynamic system. It was first proposed by Spencer, who described a game in which chips are taken from a pile and split as evenly as possible among neighboring piles \cite{spencer1986balancing}. Later, classical chip-firing was defined in \cite{anderson1989disks} and \cite{bjorner1991chip}, as a system in which one takes chips from a pile and \textit{fires} one chip to each neighboring pile. Numerous variants, such as the Abelian Sandpile (see \cite{bak1987self, dhar1990self,dhar1999abelian}) and the chip-firing game on invertible matrices \cite{zbMATH06585703}, and algebraic aspects of the game, such as the critical group \cite{MR3144399}, have been studied. When the chip-firing game is done with \textit{distinguishable}, or labeled chips, numerous basic properties from the unlabeled version of the game no longer hold, motivating a distinct area of study.

\subsection{Unlabeled Chip-Firing}
Chip-firing has classically been studied in the unlabeled case. Specifically, unlabeled chip-firing is a variant of chip-firing in which indistinguishable chips are placed on the vertices of a graph. We define a \textit{configuration} as a distribution of chips over the vertices of a graph $G = (V, E)$; it can be represented as a vector $\mathcal{C} \in \mathbb{Z}_{\geq 0}^{|V|}$ where the entry corresponding to vertex $m$ is the number of chips at vertex $v_m$. If a vertex has as at least as many chips as neighbors, it can fire and give a chip to each of its neighbors. This process continues, firing as many vertices as possible. If no vertex can fire, then the process has terminated and we say that it has reached a \textit{stable configuration}, which is a configuration of chips where no vertex can fire. 

\begin{example}

Consider the binary tree below in Figure \ref{fig:perfectchips} where the top vertex is a neighbor of itself. Initially, 7 chips are placed at the root. After the root fires three times, the root has 1 chip, and each of its children has 3 chips. The children of the root will then each fire once, resulting in 3 chips at the root and 1 chip at each vertex in layer 3. Lastly, the root fires once, reaching a stable configuration with 1 chip on each vertex in layers 1 through 3.

\begin{figure}[H]
    \centering
    \includegraphics[width=0.9\linewidth]{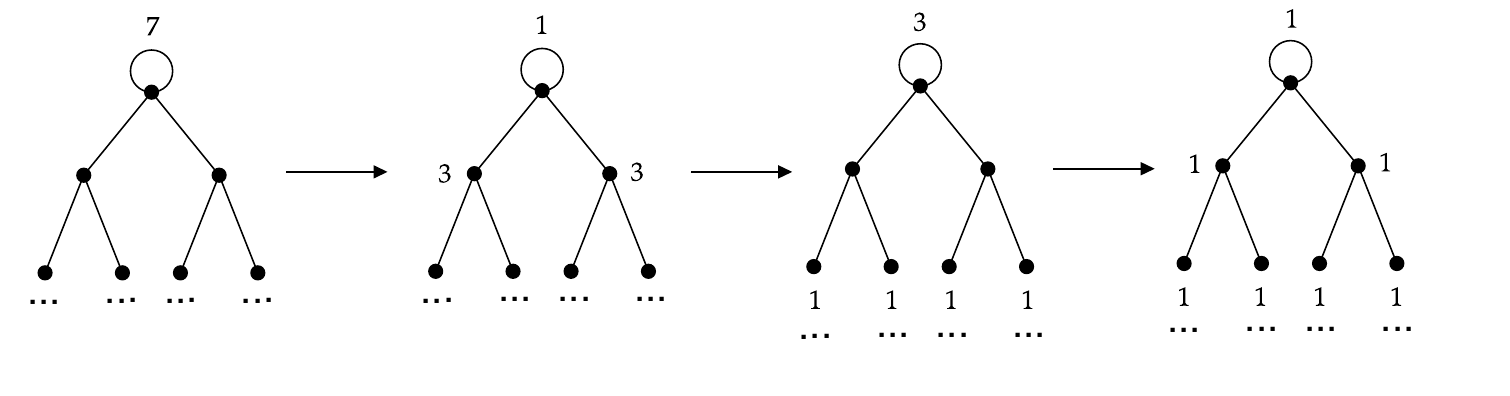}
    \caption{An example unlabeled chip-firing process on a binary tree starting with 7 chips at the root.}
    \label{fig:perfectchips}
\end{figure}

\end{example}

 One important property of unlabeled chip-firing is a commutativity property known as \textit{confluence}, which we formally state below.
\begin{theorem}[Theorem 2.2.2, \cite{klivans2018mathematics}]\label{thm:Confluence}

Consider configuration $\mathcal{C}$ and stable configuration $\mathcal{C}'$. If $\mathcal{C}'$ can be reached from $\mathcal{C}$ after finitely many legal firings, then $\mathcal{C}'$ is the unique stable configuration that is obtained from $\mathcal{C}$.
\end{theorem}

\subsection{Labeled-chip firing}

In this paper, we study a variant of the above game titled \textit{labeled chip-firing}, where the chips are numerically labeled $1, 2, \ldots, N$, and a specific firing mechanism is used to determine the direction in which chips are fired. This variant of chip-firing was first studied by Hopkins, McConville, and Propp in \cite{Hopkins_2017} in the context of the one-dimensional lattice. We specifically focus on labeled chip-firing on undirected $k$-ary trees.

Here, we assume that the root vertex has a self-loop and every vertex has $k$ children. In this setting, when a vertex $v$ has at least $k+1$ chips, it can fire, i.e., select a set of $k+1$ chips on it and send one to each of its neighbors. The $\lceil \frac{k+1}{2} \rceil$th-smallest of the selected chips will be sent to the parent, and of the remaining chips, the $i$th smallest chip gets sent to the $i$th left child. 

\begin{example}
    Consider the undirected binary tree starting with $7$ chips at the root. Figure~\ref{fig:SingleFiringProcess} illustrates one possible procedure that stabilizes the tree. First, we fire chips $5, 6, 7$ from the root. Then we fire $3, 4, 6$ from the root. Then we fire $1, 2$, and $4$. Afterwards, we fire $1, 3, 5$ from the left child of the root, and we fire $4, 6, 7$ from the right child of the root. Finally, we fire chips $2, 3, 6$ from the root. We are left with one chip on each vertex in the first three layers of the tree. Since no vertex has at least $3$ chips, we have reached a stable configuration.

    \begin{figure}[H]
    \centering
\includegraphics[width=0.75\linewidth]{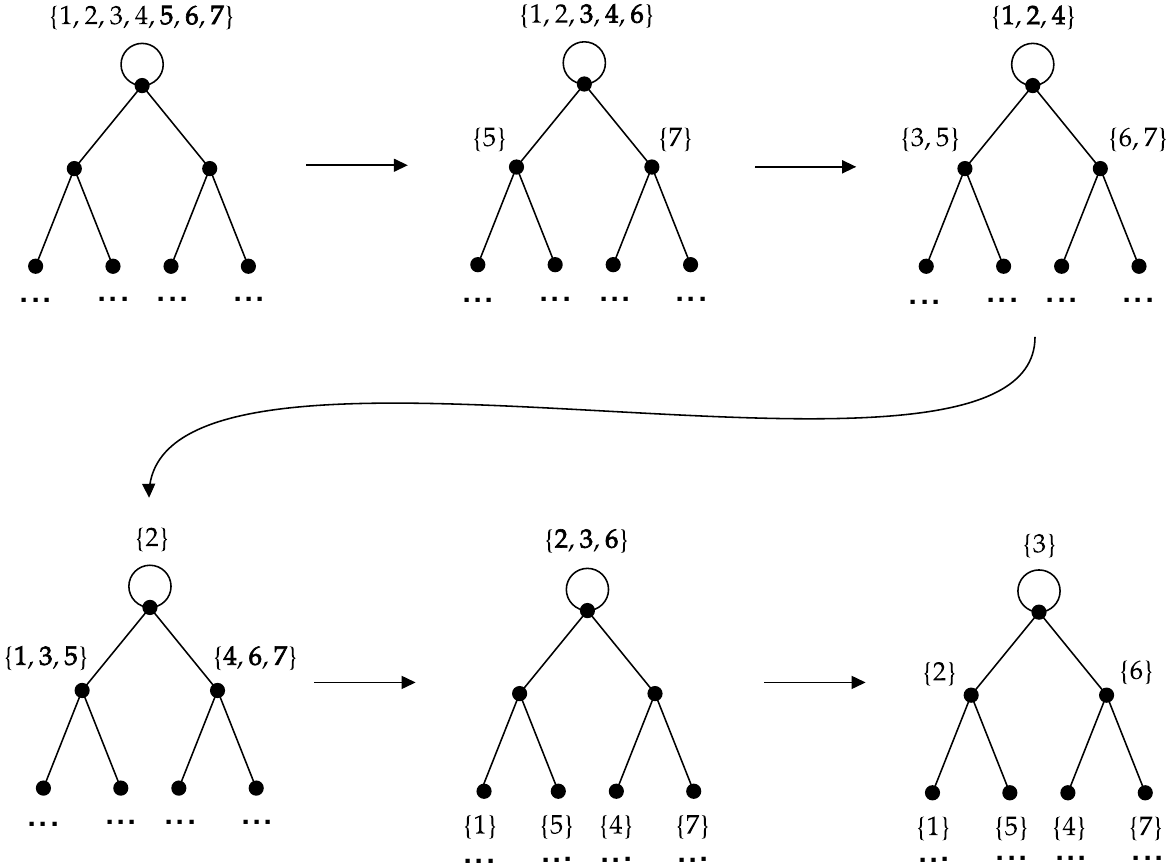}
    \caption{A labeled chip-firing process on a binary tree with 7 chips at the root. At each step, we bold the chips that get fired from the vertex.}
    \label{fig:SingleFiringProcess}
\end{figure}
    
\end{example}

One key property of labeled chip-firing on $k$-ary trees is that, unlike in unlabeled chip-firing, confluence does not always hold due to the distinguishable chips and firing mechanism; see Example~\ref{ex:nonconfluence}. Because of this, there are multiple possible stable configurations, even when we start with the same chips at the root of the tree. This lack of global confluence motivates research on labeled chip-firing.

\begin{example}\label{ex:nonconfluence}
Figure \ref{fig:labeledexample} shows an example of confluence not holding for a binary tree starting with chips $1,2, \dots 2^3-1$. The figure shows two distinct stable configurations can be reached from the same initial configuration, demonstrating that confluence does not hold.
\begin{figure}[H]
    \centering
    \includegraphics[width=0.8\linewidth]{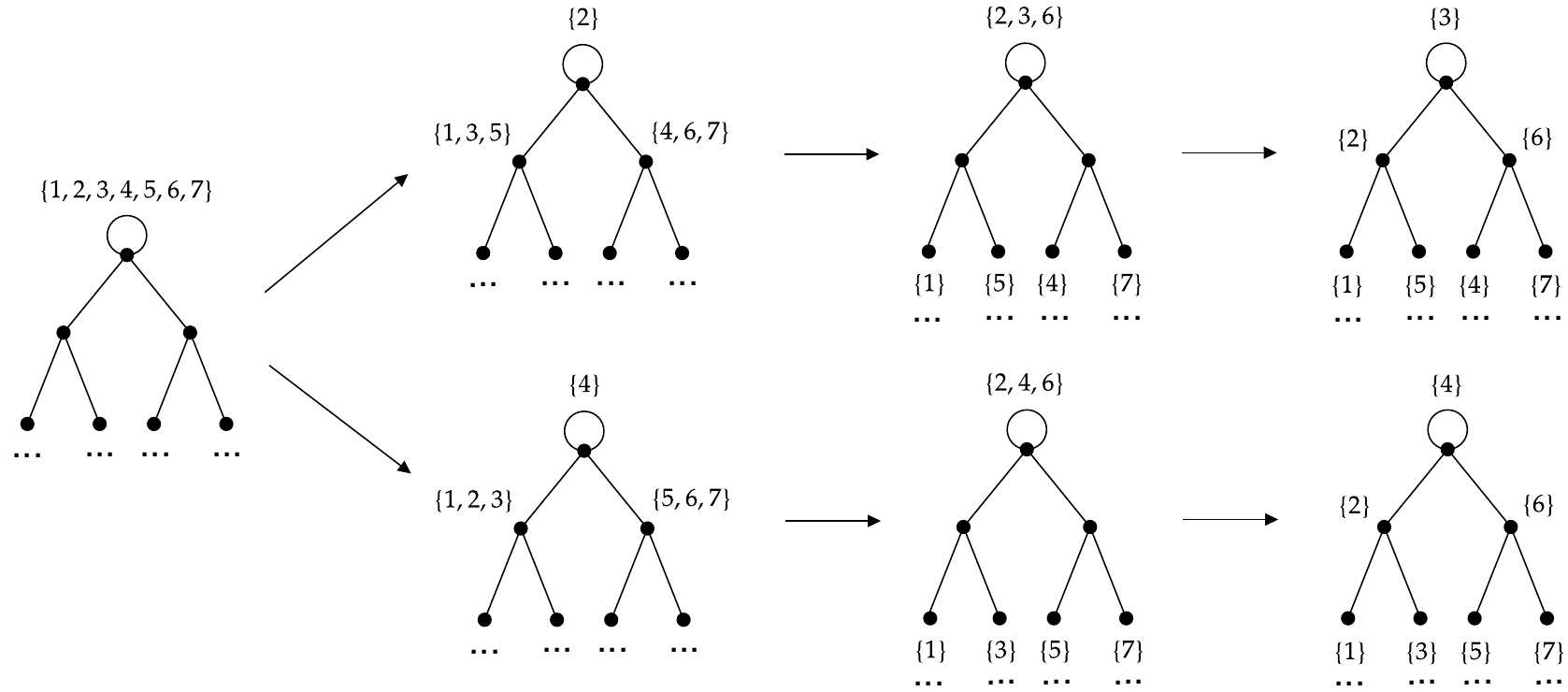}
    \caption{Two possible labeled chip-firing processes starting from the same stable configuration and ending with different stable configurations. This demonstrates that confluence, as in Theorem~\ref{thm:Confluence}, does not always hold in the labeled chip-firing game on undirected $k$-ary trees.}
    \label{fig:labeledexample}
\end{figure}
\end{example}

\subsection{Motivations and Roadmap}
In \cite{MR4827886}, Musiker and Nguyen study the labeled chip-firing game on the undirected binary tree starting with $2^{\ell}-1$ chips at the root, i.e., the $k=2$ case, and describe relations between chips that hold for all possible stable configurations. At the end of their paper \cite{MR4827886}, Musiker and Nguyen ask the following:
\begin{itemize}
    \item What are the attainable stable configurations?
    \item How many are there?
\end{itemize} In \cite{inagaki2024chipfiringundirectedbinarytrees}, the first author, Khovanova, and Luo partially answered the second question in the context of binary trees by proving an upper bound on the number of stable configurations.

In this paper, we answer the second question, but in the more general setting of $k$-ary trees starting with $\frac{k^{\ell}-1}{k-1}$ chips at the root. We begin in Section~\ref{sec:priorResults} by providing formal notations used to describe vertices of $k$-ary trees and describing prior results about unlabeled and labeled chip-firing on binary trees. Next, in Section~\ref{sec:UpperBounds}, we prove an upper bound on the number of stable configurations in the $k$-ary tree; we do this by generalizing the method used by the first author, Khovanova, and Luo \cite{inagaki2024chipfiringundirectedbinarytrees}. In Section~\ref{sec:LowerBounds}, we provide a lower bound on the number of stable configurations. We conclude the paper with conjectures and further questions in Section~\ref{sec:Further}.
\section{Preliminaries and Prior Results}\label{sec:priorResults}

\subsection{Definitions and Notations}

We begin with fundamental definitions about trees and then expand into chip-firing on trees.

A \textit{tree} is an acyclic, undirected graph. A \textit{rooted tree} is a tree with a designated vertex as the root. For a vertex $u$, the unique vertex $v$ such that $u$ and $v$ are directly connected by an edge and $v$ is closer to the root than $u$ is the \textit{parent} of $u$. Since multiple vertices can have the same parent, we denote the collection of vertices that have a vertex $v$ as its parent as the \textit{children} of $v$. In \textit{infinite looped $k$-ary tree} is an infinite tree such that every non-root vertex has $k$ children, and the root has a self-loop. This ensures that each vertex has degree $k+1$. 

For the rest of this paper, assume all trees are infinite. Furthermore, assume that all $k$-ary trees have a self-loop at the root. 

Next, we describe how we denote each vertex. We utilize a breadth-first index on the vertices. We begin by indexing the root as vertex $v_0$ and for each $i$, we label the $j$th leftmost child of $v_i$ as $v_{ki+j}$. We illustrate this labeling in Figure \ref{fig:k-ary_vertex_label}.

\begin{figure}[H]
    \centering
    \includegraphics[width=0.65\linewidth]{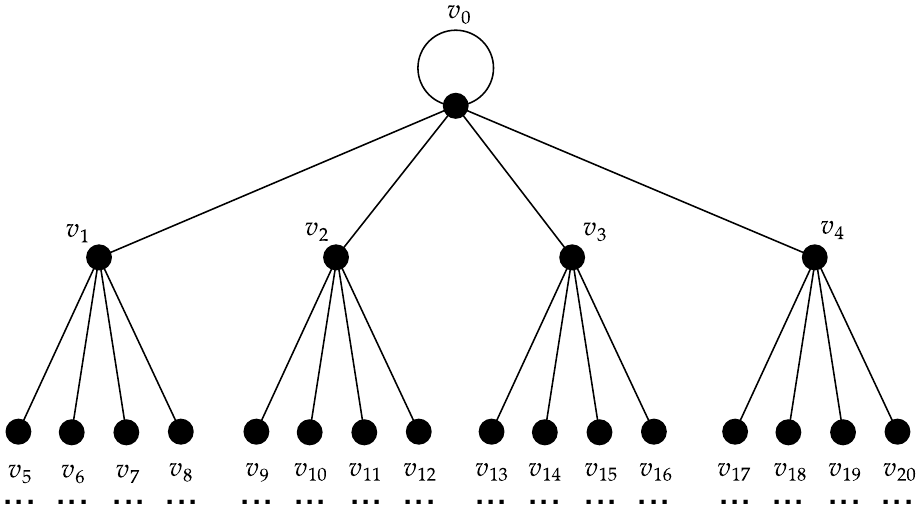}
    \caption{Vertex labelings for the first three layers of a 4-ary tree.}
    \label{fig:k-ary_vertex_label}
\end{figure}

For index $i$ and $a \in \{1, 2, \dots, k\}$, we say that $v_{ki+a}$ is a $\textit{left child}$ of $v_i$ if $a \in \{1, 2, \dots, \lfloor \frac{k}{2} \rfloor\}$.  Similarly, for each index $i$ and $a \in \{1, 2, \dots, k\}$, the vertex $v_{ki+a}$ is a \emph{right child} of $v_i$ if $a \in \{\lfloor \frac{k}{2} \rfloor+1, \lfloor \frac{k}{2}\rfloor+2, \dots, k\}$.

In this paper, we assume that our initial configuration of chips consists of labeled chips $1, 2, \dots, \frac{k^{\ell}-1}{k-1}$ on the root vertex $v_0$. We say that a vertex $v$ of this $k$-ary tree can \emph{fire}, i.e., choose $k+1$ chips and distribute a chip to each child and parent as follows: send the  $\lceil \frac{k+1}{2} \rceil$th-smallest of the selected chips to the parent and, from of the $k$ remaining chips, the $i$th smallest chip gets sent to the $i$th leftmost child. A \emph{stable configuration} is a configuration of chips where no vertex in the $k$-ary tree can fire.

We now introduce terminology used to describe relationships between vertices in our $k$-ary tree. We say that vertex $u$ is a \textit{straight left descendant} of vertex $v$ if $u$ can be reached by traveling only along the leftmost children starting at $v$. Similarly, $u$ is a \textit{straight right descendant} of vertex $v$ if $u$ can be reached by traveling only along the rightmost children starting at $v$.

\subsection{Prior Results}

The property of confluence tells us that whether a configuration stabilizes or not only depends on the initial chip configuration, not the order of firings. Next, we describe some known results about unlabeled chip-firing on trees. We begin with the number of chips on a vertex in each layer in the stable configuration of a $k$-ary tree.

\begin{theorem}[Agrawal et al. \cite{agrawal2025chipfiringinfinitekarytrees}]\label{endingconfig}

If we start with $N$ chips at the root, where $\frac{k^\ell-1}{k-1} \le N < \frac{k^{\ell+1}-1}{k-1}$ for some integer $\ell$, then for $0 \le i \le \ell-1$, the resulting stable configuration has $a_i+1$ chips on each vertex on layer $i+1$, where $a_{\ell-1} \ldots a_2 a_1 a_0$ is the base-$k$ expansion of $N-\frac{k^\ell-1}{k-1}$.
\end{theorem}

The above theorem implies that if we start with exactly $\frac{k^\ell-1}{k-1}$ chips at the root of a $k$-ary tree, then the stable configuration will have exactly 1 chip on each vertex in layers 1 through $\ell$. Because of this nice stable configuration, we choose to study stable configurations of $k$-ary trees starting with $\frac{k^\ell-1}{k-1}$ labeled chips. We use $N_{k, \ell}$ to denote \[N_{k, \ell}:= \frac{k^\ell-1}{k-1}\] for the sake of abbreviation. The above theorem is exemplified in Figure \ref{fig:perfectchips}, where each vertex in layers 1 through 3 ends with exactly 1 chip in the stable configuration.


We now provide definitions that we shall use to first state the results found in \cite{MR4827886} about labeled chip-firing on binary trees and then extend them to the setting of undirected $k$-ary trees.

\begin{definition}
Consider a stable configuration of $\frac{k^n-1}{k-1}$ chips on a $k$-ary tree. The \textit{bottom straight left descendant} and \textit{bottom straight right descendant} are the straight left and right descendants at the lowest layer such that they have at least 1 chip in the stable configuration, respectively. The \textit{top straight ancestor} of a vertex $v$ is the vertex closest to the root such that $v$ is a straight left or right descendant of its top straight ancestor.
\end{definition}

\begin{example}
    Consider the stable configuration of the undirected binary tree illustrated in Figure~\ref{fig:SingleFiringProcess}. The bottom straight left descendant of the root (the vertex with chip $3$) is the vertex with chip $1$. The bottom straight right descendant of the root is the vertex with chip $7$. 
\end{example}

\begin{example}
    Consider the stable configuration of the undirected $4$-ary tree illustrated in Figure~\ref{fig:descendant_example}.We obtain the first configuration by firing $(1, 5, 9, 14, 18)$, $(2, 6, 10, 15, 19)$, $(3, 7, 11, 16, 20)$, $(4, 8, 12, 17, 21)$, $(9, 10, 11, 12, 13)$, $(1, 2, 3, 4, 9)$, $(5, 6, 7, 8, 10)$, $(12, 14, 15, 16, 17)$, and $(13, 18, 19, 20, 21)$. The bottom straight left descendant of the root (the vertex with chip $11$) is the vertex with chip $1$. The bottom straight right descendant of the root is the vertex with chip $21$. 
\end{example}

\begin{figure}[H]
    \centering
    \includegraphics[width=0.6\linewidth]{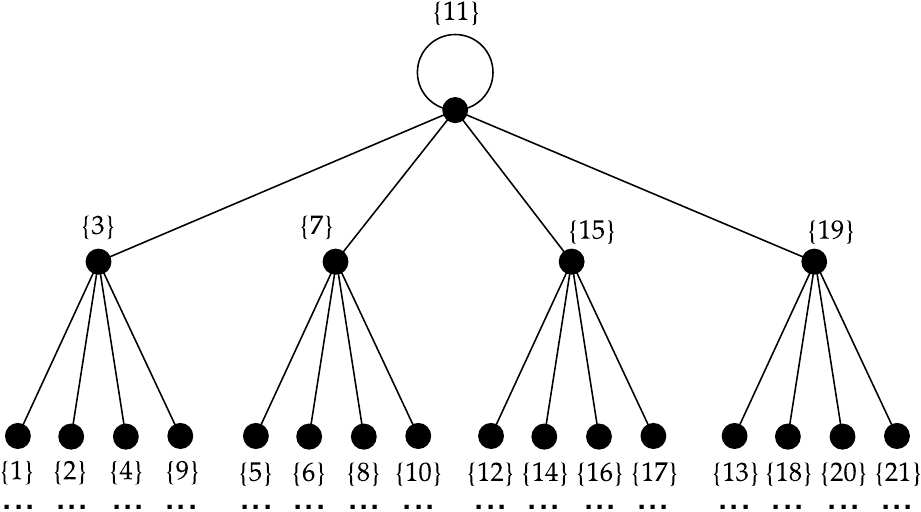}
    \caption{A stable configuration on a 4-ary tree.}
    \label{fig:descendant_example}
\end{figure}

We now proceed with the relevant results in \cite{ inagaki2024chipfiringundirectedbinarytrees, MR4827886}.

\begin{prop}[Proposition 4.4 in \cite{MR4827886}]\label{4.4}

For any vertex whose top straight ancestor is its parent, if the vertex is a left child, then the chip on it is less than the chips on its parent and its right sibling. The opposite result for a right child also holds.

\end{prop}

\begin{prop}[Proposition 4.5 in \cite{MR4827886}]\label{4.5}

The vertices that are parents of the bottom vertices contain the smallest and largest chips among the subtree of their straight ancestors, excluding the vertices on the bottom layer.

\end{prop}

\begin{lemma}[Lemma 3.1 in \cite{inagaki2024chipfiringundirectedbinarytrees}]

Suppose we start with $2^\ell-1$ labeled chips at the root of a binary tree. Then, in the stable configuration, chip 2 is located on the parent of the vertex containing chip 1, and chip $2^\ell-2$ is on the parent of the vertex containing chip $2^\ell-1$.

\end{lemma}

Using the above properties, a bound was then constructed in \cite{inagaki2024chipfiringundirectedbinarytrees} on the number of possible stable configurations for a binary tree. For labeled chip-firing on a $k$-ary tree, let $Z_{k, \ell}$ denote the number of possible stable configurations with $N_{k, \ell}$ chips at the root and let $T_{k, \ell}$ denote the number of possible orderings of the chips within a subtree consisting of $\ell$ layers in the stable configuration. We now state known bounds for $Z_{2, \ell}$ and $T_{2, \ell}$. First, let $A_\ell$ denote the \textit{Euler zigzag numbers}, or the number of alternating permutations of the set $\{1,2,\ldots, \ell \}$. Further, we define $\beta_\ell$ and $\gamma_\ell$ as a shorthand for the following expressions:

$$\zeta_\ell = A_\ell \binom{2^\ell-3}{\ell} \binom{2^\ell-\ell-3}{2^{\ell-1}-2, 2^{\ell-2}-2, 2^{\ell-3}-1, 2^{\ell-4}-1, \ldots, 1}$$ and $$\gamma_\ell = A_\ell \binom{2^\ell-5}{\ell} \binom{2^\ell-\ell-5}{2^{\ell-1}-3, 2^{\ell-2}-3, 2^{\ell-3}-1, 2^{\ell-4}-1, \ldots, 1}.$$

We then have the following bound, known as the zigzag bound.

\begin{theorem}[Theorem 3.6 of \cite{inagaki2024chipfiringundirectedbinarytrees}]\label{thm:BinaryIKL} If there are $N_{2, \ell}$ labeled chips at the root of a binary tree with $\ell \ge 4$, then $$T_{2,\ell} \le 10^{2^{\ell-4}} \zeta_\ell \prod_{i=4}^{\ell-1} \zeta_i^{2^{\ell-1-i}}$$ and $$Z_{2, \ell} \le 10^{2^{\ell-4}} \gamma_\ell \prod_{i=4}^{\ell-1} \zeta_i^{2^{\ell-1-i}}.$$

\end{theorem}

\section{An Upper Bound on the Number of Stable Configurations}\label{sec:UpperBounds}

We now derive an upper bound on the number of stable configurations of chips on the undirected $k$-ary tree starting with $N_{k, \ell}$ labeled chips at the root.

\subsection{Generalizing the Endgame of Musiker and Nguyen}

We now expand known results from \cite{inagaki2024chipfiringundirectedbinarytrees} and \cite{MR4827886} on binary trees to $k$-ary trees in general. We first begin by finding the set of firings at the end of the labeled chip-firing process where confluence holds on $k$-ary trees. We refer to this set as the \textit{Endgame}. The Endgame on binary trees was used in \cite{MR4827886} to prove Proposition \ref{4.4} and Proposition \ref{4.5}.

First, let $(i, j)$ be the $j$th-to-last fire of vertex $v_i$, where the 0th-to-last fire of a vertex denotes the last time that vertex fires. Next, define the relation $(i_1, j_1) <  (i_2, j_2)$ if $(i_1, j_1)$ must occur before $(i_2, j_2)$. This then creates a poset on the possible moves in a labeled chip-firing process. We now define the Endgame similarly to how it was done in \cite{MR4827886}.

\begin{definition}
We call the set of moves $(i,j)$ where $\mathrm{layer}(i)+j < \ell$ the \textit{Endgame}.
\end{definition}

Next, we show that during the Endgame, there is a strict ordering between when the parent and its children fire; this is our generalization of Proposition 4.2 of \cite{MR4827886}. This strict ordering will then allow us to prove that confluence holds on the endgame. Note that for a $k$-ary tree, if a non-root vertex has label $s$, its children have labels $ks+1, ks+2, \ldots, ks+k$ and its parent has label $\left \lfloor \frac {s-1} k\right \rfloor$.

\begin{theorem}\label{thm:Endgame}
Consider a vertex $i$ with $\mathrm{layer}(i) > 1$. Then for all $j$ such that $\mathrm{layer}(i)+j < \ell$, we have:

\begin{itemize}
    \item $\left( \left \lfloor \frac {i-1} k\right \rfloor, j \right) < (i,j) < \left( \left \lfloor \frac {i-1} k\right \rfloor, j+1 \right)$
    \item $(i,j)$ occurs when vertex $i$ has $k+1$ chips.
\end{itemize}
\end{theorem}

\begin{proof}
We prove the above proposition by induction on both $\mathrm{layer}(v_i)$ and $j$. We do this by proving the following steps:
\begin{enumerate}
    \item We have $(0,0) < (a, 0)$ for all $a \in \{1, 2, \dots, k\}$.
    \item We have $(i,0) < ({ki+a},0)$ for all $a \in \{1, 2, \dots, k\}$, assuming $\left ( \left \lfloor \frac {i-1} k \right \rfloor , 0 \right ) < (i,0)$.
    \item For all $a \in \{1, 2, \dots, k\}$, we have $(ki+a, 0) < (i,1)$.
    \item We have $(0,j) < (a, j)$ for all $a \in \{1, 2, \dots, k\}$, assuming inductively $(a',j-1)< (0,j)$ for all $a' \in \{1, 2, \dots, k\}$.
    \item We obtain $(i,j) < ({ki+a'}, j)$ for all $a' \in \{1,2,\dots, k\}$, assuming inductively $\left ( \left \lfloor \frac {i-1} k \right \rfloor , j \right ) < (i,j)$ and $(ki+a, j-1) < (i,j)$ for all $a \in \{1,2, \dots, k\}$
    \item We have $(i,j) < \left ({\left \lfloor \frac {i-1} k \right \rfloor}, j+1 \right)$, assuming inductively $({ki+a}, j-1) < (i,j)$ for all $a \in \{1,2, \dots, k\}$.
\end{enumerate}

When steps 1, 2, and 3 are proven, we have proven the base case. Proving steps 4, 5, and 6 proves the inductive step. 

We now prove each step of the induction.

\begin{enumerate}
    \item Suppose $(0,0)$ occurs before $(1,0)$. Then, after $(0,0)$, vertex 1 will have at least 1 chip. Further, $(1,0)$ will then increase the number of chips at $(0,0)$ by another chip, so it will have at least 2 chips at the end, a contradiction. Similarly, we know that $(0,0) < (a,0)$ for all $a \in \{2, 3, \dots, k\}$. Therefore, $(0,0)$ must occur when vertex 0 has $k+1$ chips, as it needs to end with 1 chip.

    \item Suppose $(i,0)$ occurs before $(ki+1, 0)$. Then after $(i,0)$ occurs, vertex $i$ will gain two chips from $(ki+1, 0)$ and $\left ( \left \lfloor \frac {i-1} k \right \rfloor, 0 \right )$, a contradiction. Similarly, we know that $(i,0) < (ki+a,0)$ for all $a \in \{2, 3, \dots, k\}$. Further, since vertex $i$ must end with 1 chip, $(i,0)$ must occur when vertex $i$ has $k+1$ chips.

    \item Suppose $(ki+1,0)$ occurs before $(i,1)$. Then after $(ki+1, 0)$, vertex $ki+1$ will gain two chips from $(i,1)$ and $(i,0)$, a contradiction. Therefore, $(ki+a, 0) < (i, 1)$ for all $a \in \{1, 2, \dots, k\}$.

    \item Suppose $(0,j)$ occurs before $(1,j)$. Then since $(s,j-1)< (0,j)$ for all $s \in \{1,2, \dots, k\}$, after $(0,j)$, vertex 0 will keep at least 1 chip, lose $k(j-1)$, gain at least $j$ chips from vertex 1, and gain at least $(k-1)(j-1)$ chips from vertices $2, 3, \ldots, k$. Then vertex 0 will have at least 2 chips, a contradiction. Similarly, we know that $(0,j) < (a,j)$ for all $a \in \{1, 2, \dots, k\}$. Therefore, after $(0,j)$, vertex 0 will lose $k(j-1)$ chips and gain $k(j-1)$, meaning $(0,j)$ must occur when vertex 0 has $k+1$ chips.

    \item Suppose $(i,j)$ occurs before $(ki+1, j)$. Then after $(i,j)$ vertex $i$ will lose $(k+1)(j-1)$ chips. Further, since $\left ( \left \lfloor \frac {i-1} k \right \rfloor, j \right ) < (i,j)$, vertex $i$ will receive at least $j$ chips from vertex $\left \lfloor \frac {i-1} k \right \rfloor$, $j$ chips from vertex $ki+1$, and $(k-1)(j-1)$ chips from vertices $ki+2, \ldots, ki+k$. Then vertex $i$ will have at least 2 chips at the end, a contradiction. Further, since $(i,j) < (ki+a', j)$ for all $a' \in \{1,2, \dots, k\}$, after $(i,j)$, vertex $i$ will lose $(k+1)(j-1)$ chips and gain $j+k(j-1) = (k+1)j-k$ chips, so $(i,j)$ must occur when vertex $i$ has $k+1$ chips.

    \item Suppose $(i,j)$ occurs before $\left(\left \lfloor \frac {i-1} k \right \rfloor, j+1 \right)$. Then after $(i,j)$, vertex $i$ will lose $(k+1)(j-1)$ more chips. However, since $(ki+a, j-1) < (i,j)$ for all $a \in \{1, 2, \dots, k\}$, vertex $i$ will gain at least $k(j-1)$ chips from vertices $ki+1, \ldots, ki+k$ and also gain $j+1$ chips from vertex $\left \lfloor \frac {i-1} k \right \rfloor$. Therefore, vertex $i$ receives at least $(k+1)j-1$ chips, leaving vertex $i$ with at least 2 chips at the end, a contradiction. Therefore, $(i,j) < \left(\left \lfloor \frac {i-1} k \right \rfloor, j+1 \right).$
\end{enumerate}
\end{proof}

Next, we show that confluence holds on the endgame of a labeled chip-firing process on a $k$-ary tree; this is our generalization of Corollary 4.3 of \cite{MR4827886}.

\begin{corollary}[Endgame Confluence]
For a looped $k$-ary tree with $N_{k, \ell}$ chips at the root, once $(0, \ell-2)$ has occurred, the order in which vertices fire does not affect the stable configuration.
\end{corollary}

\begin{proof}
During the endgame, a vertex fires for the $j$th to last time before its parent fires for the $j$th to last time but after its parent fires for the $(j+1)$th to last time. Therefore, if a vertex can fire, none of its children or parents can fire. Further, since a vertex can only fire when it has exactly $k+1$ chips, the firing cannot be affected by any of its neighbors. Therefore, the order does not affect the stable configuration.
\end{proof}

For example, Figure \ref{fig:endgameexample} is a case of the endgame on a 4-ary tree. Notice that there is only 1 possible stable configuration, as confluence holds on the endgame. Furthermore, since confluence holds on the endgame, we now consider a specific order of firings. First, notice that if we have $N_{k, \ell}$ chips at the root, then the endgame starts when the root has $k+1$ chips, each vertex on layers 2 to $\ell-1$ has $k$ chips, and each vertex on layer $\ell$ or lower has 0 chips. 

\begin{figure}[h]
    \centering
    \includegraphics[width=0.9\linewidth]{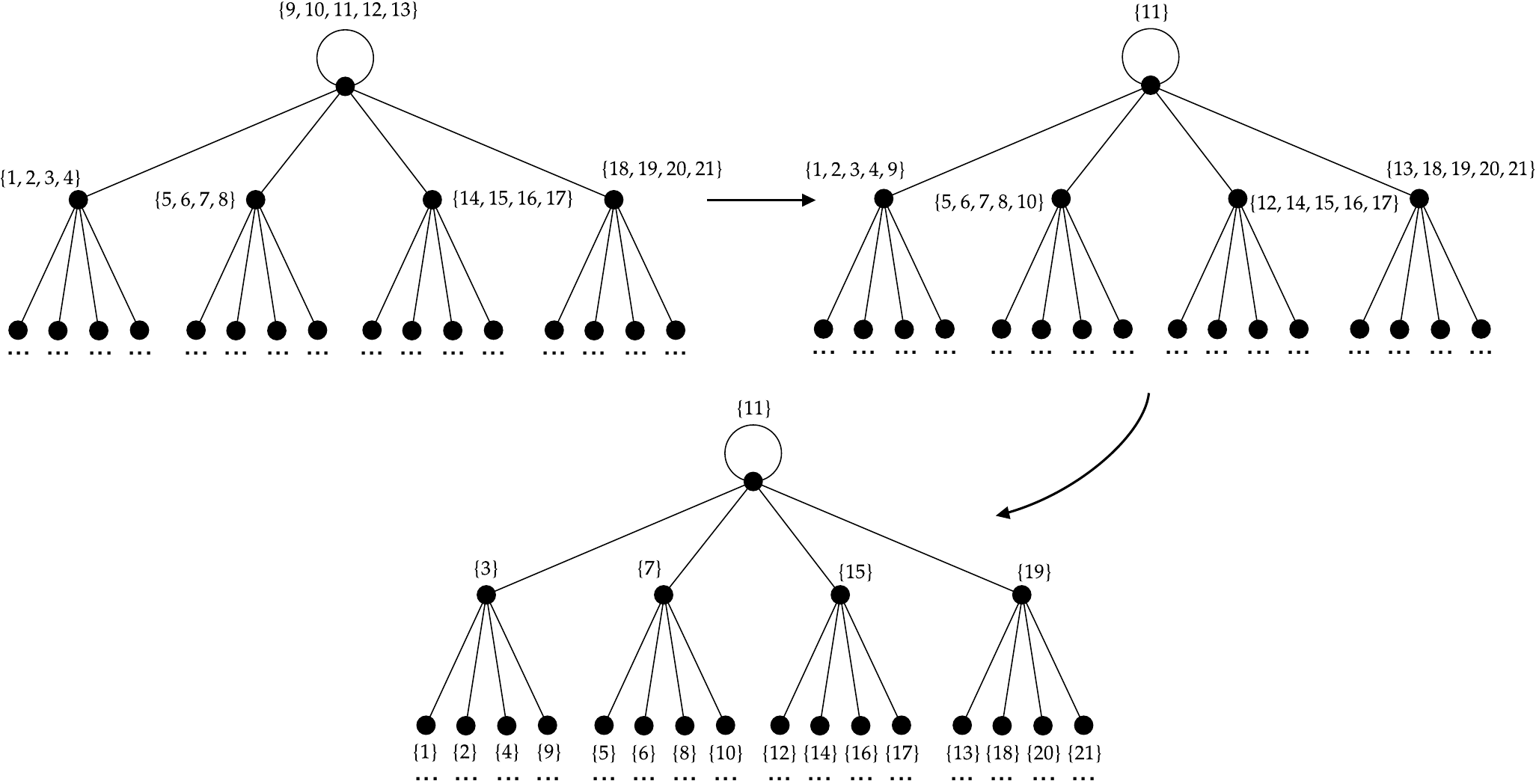}
    \caption{A labeled chip-firing process on a 4-ary tree, starting from the endgame and ending at the stable configuration.}
    \label{fig:endgameexample}
\end{figure}

Next, we assume that in the endgame, we fire in the following process: we first fire each vertex from 0 to $N_{k, \ell-1}-1$ exactly once, then we fire each vertex from 0 to $N_{k, \ell-2}-1$ exactly once and we continue so on and so forth until we fire just vertex $N_{k, \ell-(\ell-1)}-1 = 0$, the root, exactly once. 

Notice that in every step, we fire every vertex that can fire once. Denote each of these steps as a \textit{wave}, shown in Figure \ref{fig:wave}. We now use this specific firing during the endgame to prove various properties of a stable configuration and later extend the zigzag bound to $k$-ary trees in general.

\begin{figure}[H]
    \centering
    \includegraphics[width=0.8\linewidth]{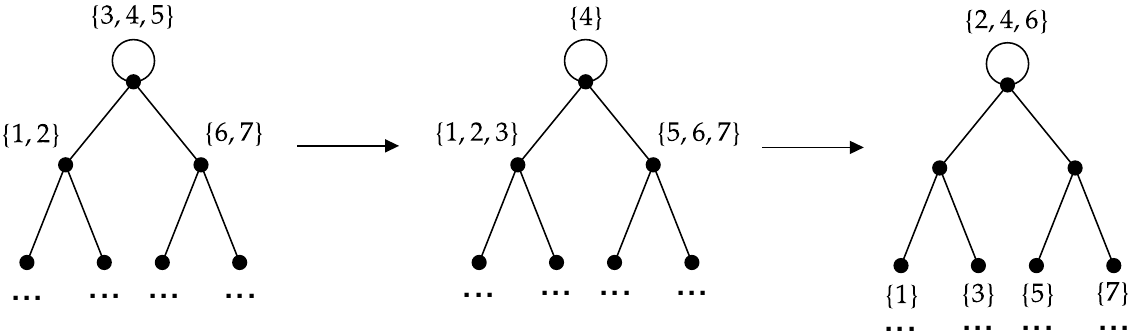}
    \caption{The first wave of the endgame for a labeled chip-firing process on a binary tree starting with 7 chips at the root.}
    \label{fig:wave}
\end{figure}

\begin{prop}\label{prop:smallestchip}
In labeled chip-firing on an infinite $k$-ary tree with $N_{k, \ell}$ labeled chips at the root, the bottom left and right descendants of a vertex contain the smallest and largest chips among its subtree.
\end{prop}

\begin{proof}

We only consider the case for the bottom straight left descendant, as the bottom straight right descendant case is analogous. Further, if a vertex $v_i$ was the leftmost child of its parent, we can instead prove the statement for its parent. Therefore, we only consider the cases where $v_i$ is not the leftmost child.

First, consider the smallest chip that $v_i$ contains throughout the firing process. Call this chip $c$. Since $c$ is the smallest chip in the subtree rooted at $v_i$, whenever it is fired, it will always fire to the leftmost child. Therefore, $c$ will always stay among $v_i$ and its straight left descendants.

Next, from the first move of vertex $v_i$ in the endgame, $(i, \ell-1-\text{layer}(i))$, vertex $v_i$ will alternate between sending a chip to its parent and receiving a chip from its parent. Since $v_i$ isn't the leftmost child, the chip it receives will be larger than or equal to the chip it sends back, meaning $c$ cannot reach $v_i$ for the first time after $(i, \ell-1-\text{layer}(i))$. This means that $c$ is in the subtree rooted by $v_i$ during the first wave. Since $c$ can only be fired to the left, this then means that $c$ will reach the end of the first wave, becoming the bottom straight left descendant of $v_i$.
\end{proof}

This proposition leads us to the following ``naive" bound on the number of stable configurations on the undirected $k$-ary tree when starting with $N_{k, \ell}$ chips.
\begin{prop}[Naive Bound]\label{prop:Naive}
    For $\ell > 2$, we have $Z_{k, \ell}$ and $ T_{k, \ell}$ are both at most $ (N_{k, \ell}-2)!$.
\end{prop}
\begin{proof}
    For any undirected $k$-ary tree or any subtree of one, with $N_{k, \ell}$ labeled chips, Proposition~\ref{prop:smallestchip} guarantees the locations of the smallest and largest of those chips. Thus, ordering the remaining chips, there are at most $(N_{k, \ell}-2)!$ possible stable configurations.
\end{proof}

In \cite{inagaki2024chipfiringundirectedbinarytrees}, it was found that, naively, $Z_{2, \ell} \leq (N_{2, \ell}-5)!$. This is because, in the context of the stable configuration of the binary tree with $2^{\ell}-1$ labeled chips, chip $2$ is always at the parent of the vertex containing chip $1$ and $2^{\ell}-2$ is always in the parent of the vertex containing chip $2^{\ell}-1$ (see Lemma 3.1 of \cite{inagaki2024chipfiringundirectedbinarytrees} for a precise statement). However, for stable configurations of $k$-ary trees with $k \geq 4$,we find that the second smallest chip is not necessarily at the parent of the bottom straight left descendant of the root. For instance, as illustrated in Figure~\ref{fig:chip2_position}, chip $2$ can end up at the same level as the bottom straight left descendant of the root, and chip $N_{k, \ell}-1 = \frac{4^3-1}{4-1}-1 = 20$ can end up at the same level as the bottom straight right descendant of the root.

\begin{figure}[H]
    \centering
    \includegraphics[width=0.95\linewidth]{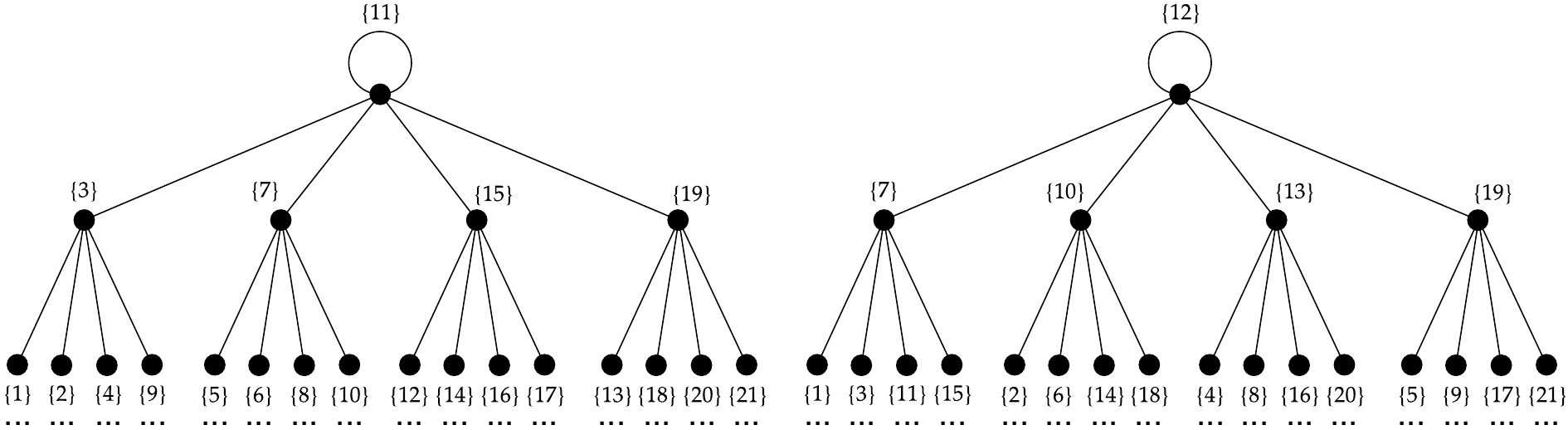}
    \caption{Two stable configurations on a $4$-ary tree. The stable configuration in the left tree is the one in Figure \ref{fig:endgameexample}, and the sequence if tuples of chips fired to reach the stable configuration in the right tree is $(1,2,3,4,5)$, $(3,6,7,8,9)$, $(7,10,11,12,13)$, $(11,14,15,16,17)$, $(15,18,19,20,21)$, $(1,3,7,11,15)$, $(2,6,10,14,18)$, $(4, 8,12,16,20)$, $(5,9,13,17,21)$, $(7,10,12,13,19)$.}
    \label{fig:chip2_position}
\end{figure}

\subsection{The Zigzag Bound}
Using properties resulting from the generalization of the Endgame, we now derive a nontrivial upper bound on the number of possible stable configurations, which we call the \textit{zigzag bound}.

\begin{definition}
    A \textit{zigzag} of length $\ell$ is a sequence of vertices $v_{a_1}, v_{a_2}, \dots, v_{a_{\ell}}$ defined as follows. Establish a starting vertex $v_{a_1}$. Then the other vertices are defined as follows. Let $a_{i+1} = ka_i+k$ if $a_i \equiv r \pmod{k}$ for $r \in \{1, 2, \dots, \lfloor k/2 \rfloor\}$, and we let $a_{i+1} = ka_{i}+1$ if $a_1 \equiv r \pmod{k}$ for some $r \in \{\lfloor k/2\rfloor+1, \lfloor k/2 \rfloor +2, \dots, k\}.$ 
\end{definition}
\begin{example}
Consider a $3$-ary tree as shown in Figure~\ref{fig:3ary_zigzag}. The vertices in the zigzag and the edges connecting them are bolded in the Figure.

\begin{figure}[H]
    \centering
    \includegraphics[width=0.25\linewidth, angle=270]{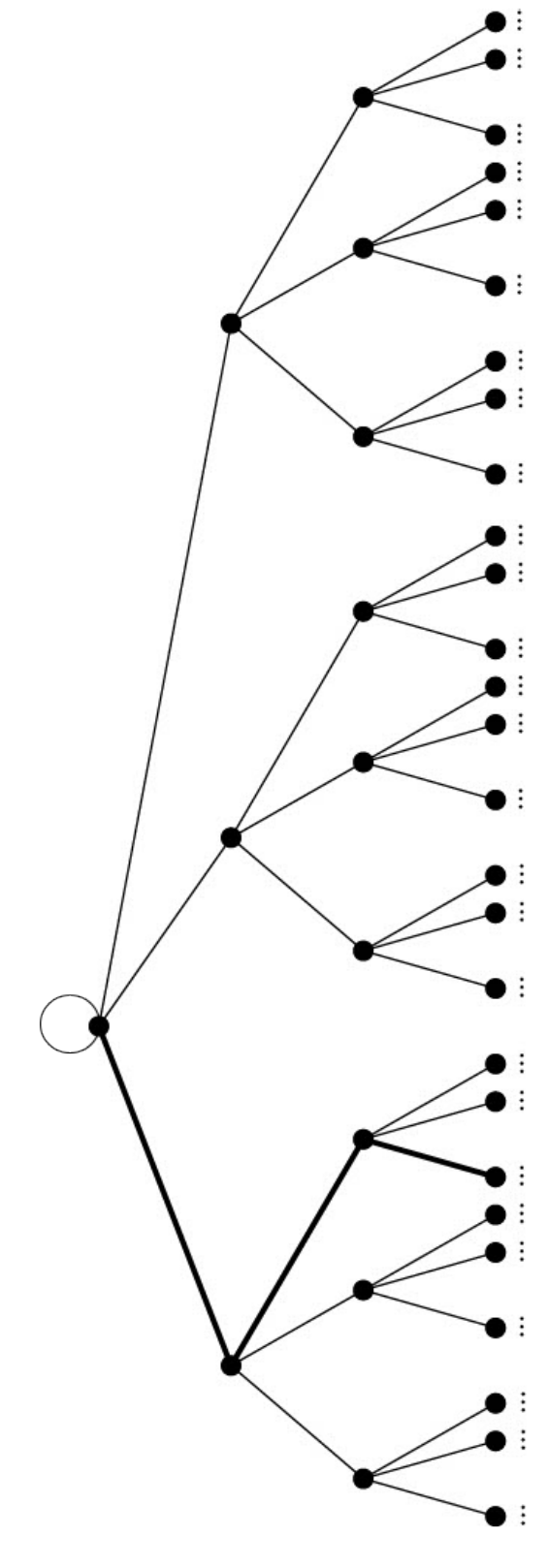}
    \caption{A $3$-ary tree with a zigzag of length 3 starting from the root.}
    \label{fig:3ary_zigzag}
\end{figure}
\end{example}

\begin{example}
    The bolded path in Figure \ref{fig:karyzigzag} is an example of a zigzag on a 4-ary tree.
\end{example}

\begin{figure}[H]
    \centering
    \includegraphics[width=0.6\linewidth]{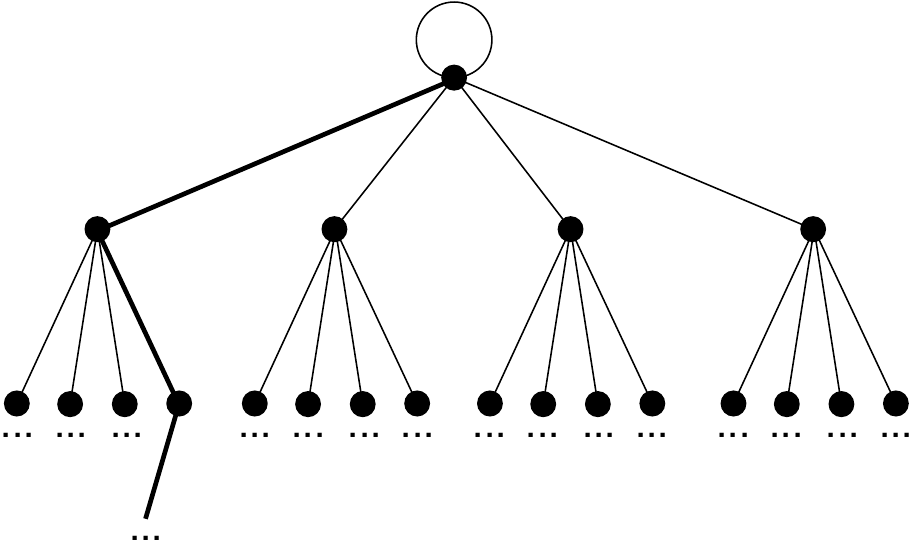}
    \caption{A 4-ary tree with a zigzag (in bold).}
    \label{fig:karyzigzag}
\end{figure}

In our proof of the zigzag bound for the $k$-ary tree starting with $N_{k, \ell}$ chips at the root, we first construct a zigzag of $\ell+1$ vertices starting at the root vertex. Then, we notice that for each vertex $v$ in the zigzag except for the last one, there are $k-1$ children of $v$ that are not in the zigzag. We therefore observe that the zigzag splits the tree into $k-1$ subtrees, each rooted at the child of the root, with $N_{k, \ell-1}$ chips, $k-1$ subtrees each rooted at the child of the second vertex, with $N_{k, \ell-2}$ chips, and so on.

\begin{example}
    In Figure~\ref{fig:3ary_zigzagpartition}, we illustrate the partitioning above for the $3$-ary tree with $N_{3, 4}$ chips at the root. Once we take a zigzag of length 4 out, we can partition the tree into two subtrees of 3 layers, two subtrees with 2 layers, and two subtrees with one layer. Each of those subtrees is rooted at children of points in the zigzag that are not in the zigzag.
    \begin{figure}[H]
    \centering
    \includegraphics[width=0.7\linewidth]{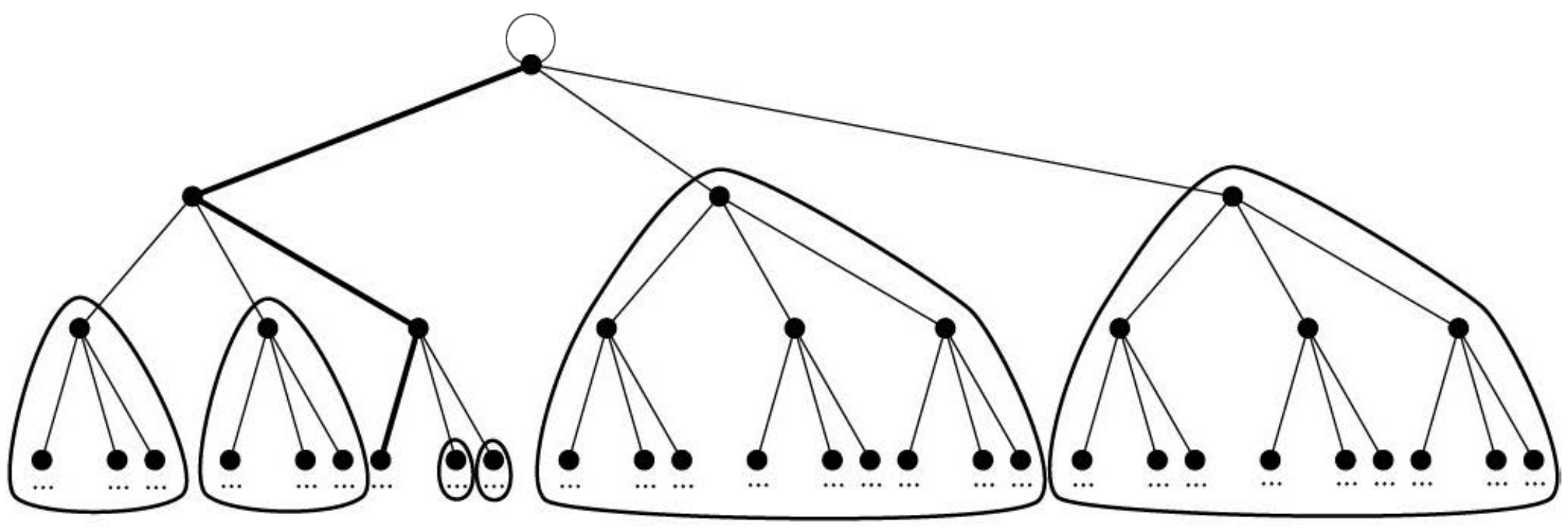}
    \caption{A partitioning of a tree starting with $N_{3, 4}$ chips at the root.}
    \label{fig:3ary_zigzagpartition}
\end{figure}
\end{example}

These zigzags have the following key property:
\begin{prop}\label{zigzagrelation}

 Let $v_s$ be a vertex and denote its label as $s$. Further, for each $i \in \{0,1, \dots\}$, let $c_i$ denote the chip at vertex $v_i$ in the stable configuration. If the vertex $v_s$ is a left child of its parent $v_{parent}$ and $v_{parent}$ is a right child of its parent, then we have the following relation $$c_{ks+1} < c_{ks+2} < \ldots < c_{ks+\lfloor k/2\rfloor } < \min \{c_s, c_{ks+\lfloor k/2 \rfloor +1}, c_{ks+\lfloor k/2 \rfloor +2}, \ldots, c_{ks+k}\}.$$ If the vertex $v_s$ is the right child of its parent $v_{parent}$ and $v_{parent}$ is a left child of its parent, we have $$c_{ks+k} > c_{ks+k-1} > \ldots > c_{ks+\lfloor k/2 \rfloor +1} > \max \{c_s, c_{ks+\lfloor k/2 \rfloor}, c_{ks+\lfloor k/2\rfloor-1}, \ldots, c_{ks+1}\}.$$

\end{prop}

\begin{proof}

We prove the case for which the vertex is a left child, and the opposite case follows analogously.

When vertex $v_s$ fires for the last time, the chip that vertex $ks+1$ receives is less than the chips of its $k-1$ siblings. Further, since vertex $v_s$ is a right child, the chip it receives from vertex $\left \lfloor \frac {m-1} k \right \rfloor$, i.e. the parent of vertex $v_s$, cannot be smaller than the chip vertex $v_s$ sent to vertex $ks+1$, so it also cannot be smaller than the chip at vertex $ks+1$.

Further, this also holds for vertices $v_{ks+2}, v_{ks+3}, \ldots, v_{ks+\lfloor k/2\rfloor} $, as during each firing of vertex $v_s$, the chip sent to vertex $v_{\left \lfloor \frac {s-1} k \right \rfloor}$ is greater than the chips sent to vertices $v_{ks+2}, v_{ks+3}, \ldots, v_{ks+\lfloor k/2 \rfloor}$. Since vertex $v_{\left \lfloor \frac {s-1} k \right \rfloor}$ cannot send back a smaller chip, the chips at vertices $v_{ks+2}, v_{ks+3}, \ldots, v{ks+\lfloor k/2 \rfloor}$ must then be smaller than the chip at vertex $v_s$.
\end{proof}

\begin{example}
By Proposition \ref{zigzagrelation}, we know that in Figure \ref{fig:zigzagexample}, $c_{4s+1} < c_{4s+2} < \min \{c_{s}, c_{4s+3}, c_{4s+4}\}.$
\end{example}

\begin{figure}[H]
    \centering
    \includegraphics[width=0.2\linewidth]{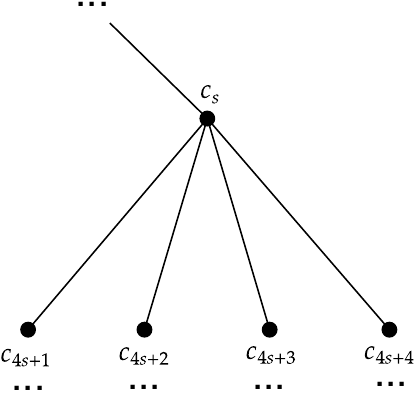}
    \caption{An example visualizing Proposition \ref{zigzagrelation} on a 4-ary tree.}
    \label{fig:zigzagexample}
\end{figure}

We now establish Lemma~\ref{zigzagrelation}, the key lemma to provide an upper bound on $Z_{k,\ell}$ and $T_{k,\ell}$.

\begin{lemma}\label{zigzagbound}

For a $k$-ary tree with $\ell \ge 3$, we have the following upper bound for $Z_{k,\ell}$ and $T_{k,\ell}$: \footnotesize{$$\binom{N_{k, \ell}-2}{\ell} \frac{(N_{k, \ell}-\ell-2)!}{(N_{k, \ell-1} - 1)!(N_{k, \ell-1})!^{k-2}(N_{k, \ell-2} - 1)!(N_{k, \ell-2})!^{k-2}(N_{k, \ell-3})!^{k-1}(N_{k, \ell-4})!^{k-1} \ldots (N_{k, 1})!^{k-1}} A_\ell \prod_{i=1}^{\ell-1} T^{k-1}_{k, i}.$$}

\end{lemma}

\begin{proof}

Consider a subtree with $N_{k, \ell}$ labeled chips at the root and a zigzag alternating between leftmost and rightmost children starting from the root.  Specifically, consider a zigzag of length $\ell$ rooted at the subtree. We begin by choosing the possible chips on the zigzag. Since $\ell \ge 3$, we know that the smallest and largest chips cannot lie on the zigzag. Therefore, we can choose the chips in the zigzag in $\binom{N_{k, \ell}-2}{\ell}$ different ways. Further, let $A_\ell$ denote the number of alternating permutations of the set $\{1, 2, \ldots, \ell\}$. Then by Proposition \ref{zigzagrelation}, there are $A_\ell$ ways to order the chips in the zigzag, meaning there are a total of $\binom{N_{k, \ell}-2}{\ell} A_\ell$ ways to select the chips on the zigzag.

Next, we consider the number of ways to order the chips on the vertices, not on the zigzag. Observe that the zigzag splits the initial subtree into $k-1$ subtrees of $\ell-1$ layers, $k-1$ subtrees of $\ell-2$ layers, $\ldots$, and $k-1$ subtrees of $1$ layer. Further, the smallest chip must end up in the leftmost subtree, and the largest chip must end up in the rightmost subtree. Therefore, instead of choosing $N_{k, \ell-2}$ and $N_{k, \ell-1}$ chips for those subtrees, respectively, we only need to pick $N_{k, \ell-2}-1$ and $N_{k, \ell-1}-1$ chips. Therefore, there are $$\frac{(N_{k, \ell}-\ell-2)!}{(N_{k, \ell-1} - 1)!(N_{k, \ell-1})!^{k-2}(N_{k, \ell-2} - 1)!(N_{k, \ell-2})!^{k-2}(N_{k, \ell-3})!^{k-1}(N_{k, \ell-4})!^{k-1} \ldots (N_{k, 1})!^{k-1}}$$ ways to split the remaining chips into the subtrees. Lastly, there are $T_{k,i}$ ways to order the chips in a subtree with $i$ layers, so there are $\prod_{i=1}^{\ell} T^{k-1}_{k, i}$ total ways to order the chips within each subtree. Multiplying all of the terms together then gives the zigzag bound.
\end{proof}

Next, we use Proposition~\ref{zigzagbound} to provide a closed form bound on $Z_{k, \ell}$ and $T_{k, \ell}$. First, we define the following expression to collapse our recursive bound into a direct bound: $$\beta_{k, \ell} = \binom{N_{k, \ell}-2}{\ell} \frac{(N_{k, \ell}-\ell-2)!}{(N_{k, \ell-1} - 1)! (N_{k, \ell-1})!^{k-2}(N_{k, \ell-2} - 1) !(N_{k, \ell-2})!^{k-2}(N_{k, \ell-3})!^{k-1} (N_{k, \ell-4})!^{k-1}\ldots (N_{k, 1})!^{k-1}} A_\ell.$$

We also establish bounds on $T_{k, 1}, Z_{k, 1}, T_{k, 2}, Z_{k, 2}$, which we shall use to establish our recursive bound:

\begin{lemma}\label{lem:small} Let $k \geq 2$. Then, 
    $Z_{k, 1} = T_{k, 1} = Z_{k, 2} = 1$ and $ T_{k, 2} \leq k-1$.
\end{lemma}
\begin{proof}
   Recall that $Z_{k, 1}, T_{k, 1}$ are the numbers of configurations/placements of chips of the $k$-ary tree and subtree of the $k$-ary tree containing $N_{k, 1} = 1$ labeled chips. There is only one way of placing them.

    We find that $Z_{k, 2} = 1$. To see this, first note that $Z_{k, 2}$ is the number of stable configurations of $N_{k, 2} = \frac{k^2 - 1}{k-1} = k+1$ chips. By the rules of firing, we obtain that the median of those $k+1$ chips will remain at the root, and all other chips will be on the children of the root and in increasing order from left to right.
    
    We now bound $ T_{k, 2}$. Let $v$ be the root of the vertex that is $v$. We obtain that the chips in the last layer of the subtree are so that 1) the smallest and largest chips in that subtree are always in the leftmost and rightmost child of the root, by Proposition~\ref{prop:smallestchip}. Because of this, we find that there are at most $\frac{k^2-1}{k-1}-2 = k-1$ potential chips that can be at the root of the subtree of 2 layers. Then note that the leftmost vertex in that layer has the smallest chip in that layer, the second leftmost vertex in that layer has the second smallest chip in that layer, and so on. To see this, one can consider the endgame described by Theorem~\ref{thm:Endgame}; here we know that the chips in the bottommost layer of the subtree were dispersed by the last firing of the root of the subtree. Therefore, we find that the number of stable configurations of the subtree with 2 layers of chips is at most the number of options for chips that can end up at the root. Therefore $T_{k, 2} \leq k-1$.
\end{proof}

Finally, we state and prove an expression for the upper bound on $Z_{k, \ell}$ and $T_{k, \ell}$, the number of configurations.

\begin{theorem}[Zigzag Bound]\label{thm:GeneralZigzagBound}

For a $k$-ary tree, we have the following upper bound for $Z_{k,\ell}$ and $T_{k,\ell}$ when $\ell \ge 3$: 

$$T_{k, \ell}\le (k-1)^{(k-1)k^{\ell-3}} \beta_{k, \ell} \prod_{i=1}^{\ell-3} \beta_{k, \ell-i}^{(k-1)k^{i-1}}.$$ and  $$Z_{k, \ell}\le (k-1)^{(k-1)k^{\ell-3}} \beta_{k, \ell} \prod_{i=1}^{\ell-3} \beta_{k, \ell-i}^{(k-1)k^{i-1}}.$$

\end{theorem}

\begin{proof}

We first observe that $Z_{k, \ell} \le T_{k, \ell}$. This is because any stable configuration of $N_{k, \ell}$ chips on a $k$-ary tree, which is a subtree of itself consisting of $\ell$ layers of chips in the stable configuration.

We prove this by strong induction. First, by Proposition \ref{zigzagbound}, we know that $$T_{k, 3} \le \binom{N_{k, 3}-2}{3} \frac{(N_{k, 3}-5)!}{(N_{k, 2} - 1)!(N_{k, 2})!^{k-2}(N_{k, 1} - 1)!(N_{k, 1})!^{k-2}} A_3(k-1)^{(k-1)k^{3-3}} = \beta_{k,3} (k-1)^{k-1}\prod_{i=1}^{0} \beta_{k, 3-i}^{(k-1)k^{i-1}}.$$

Next, for the inductive step, we assume that for all $j \in [3,\ell]$, we have $$T_{k, j} \le (k-1)^{(k-1)k^{j-3}}\beta_{k,j} \prod_{i=1}^{j-3} \beta_{k, j-i}^{(k-1)k^{i-1}}.$$ Next, by Proposition \ref{zigzagbound}, we then know that 
\begin{align*}
T_{k,\ell+1} & \le  \beta_{k,\ell+1} \prod_{j=1}^{\ell} T^{k-1}_{k,j} = (k-1)^{k-1} \beta_{k,\ell} \prod_{j=3}^{\ell} T^{k-1}_{k,j} \le (k-1)^{k-1}\beta_{k,\ell+1} \prod_{j=3}^{\ell} \beta^{k-1}_{k,j} (k-1)^{(k-1)^2k^{j-3}} \prod_{i=1}^{j-3} \beta_{k, j-i}^{(k-1)^2k^{i-1}} \\
& = (k-1)^{k-1 + (k-1)^2\sum_{j'=0}^{\ell-3}k^{j'}}\beta_{k, \ell+1} \prod_{m=3}^{\ell} \beta_{k,m}^{k-1+\sum_{j=m+1}^{\ell} (k-1)^2k^{j-m-1}}  \\ & = (k-1)^{k-1 + (k-1)^2\sum_{j'=0}^{\ell-3}k^{j'}} \beta_{k, \ell+1} \prod_{m=3}^{\ell} \beta_{k,m}^{k-1+\sum_{j=0}^{\ell-m-1} (k-1)^2 k^{j}} \\
& = (k-1)^{k-1 + (k-1)(k^{\ell-2}-1)} \beta_{k, \ell+1} \prod_{m=3}^{\ell} \beta_{k,m}^{k-1+(k-1)(k^{\ell-m-1}-1)}  \\ & = \beta_{k, \ell+1} \prod_{m=3}^{\ell-1} \beta_{k,m}^{(k-1)k^{\ell-m}} \\ & =  (k-1)^{(k-1)k^{\ell-2}}\beta_{k, \ell+1} \prod_{m=1}^{\ell-2} \beta_{k,\ell+1-m}^{(k-1)k^{m-1}} \end{align*}
\end{proof}

For fixed $k$, we define the following sequence signifying the zigzag bound: \[T_{k, zig}(\ell)  = Z_{k, zig}(\ell) = (k-1)^{(k-1)k^{\ell-3}} \beta_{k, \ell}\prod_{i=1}^{\ell-3}\beta_{k, \ell-i}^{(k-1)k^{i-1}}.\]

Like Theorem 3.12 of \cite{inagaki2024chipfiringundirectedbinarytrees}, we now prove that the zigzag bound is asymptotically better than that from Proposition~\ref{prop:Naive}.
\begin{theorem}
    For $\ell \geq 4$ and $k \geq 2$, $T_{k, zig}(\ell) < (N_{k, \ell}-3)!$.
\end{theorem}
\begin{proof}
    We do a proof by induction. First observe that for $\ell=4$, 
\begin{align*}
& Z_{k, zig}(4) \\
\ = \ & \left(\frac{(N_{k, 4}-2)!}{4! (N_{k, 3}-1)!((N_{k, 3})!)^{k-2} (N_{k, 2}-1)! (N_{k, 2}!)^{k-2}} A_4\right) \left(\frac{(N_{k, 3}-2)!}{3! (N_{k, 2}-1)! (N_{k, 2}!)^{k-2}} A_3\right)^{k-1}(k-1)^{(k-1)k} \\
\ < \ & \frac{(k^3+k^2+k-1)!}{(k^2+k)!(k^2+k+1)!^{k-2}(k!)(k+1)!^{k-2}} \left(\frac{(k^2+k-1)!}{(k!) (k+1)!^{k-2}}\right)^{k-1} (k-1)^{(k-1)k} \\
\ = \ & \frac{(k^3+k^2+k-1)! (k^2+k-1)!}{(k^2+k)! (k!)^k ((k+1)!)^{k^2-2k}} \left(\frac{(k-1)^{k(k-1)}}{((k^2+k+1)(k^2+k))^{k-2}}\right) \\
\ < \ & \frac{(k^3+k^2+k-1)! (k^2+k-1)!}{(k^2+k)! (k!)^k ((k+1)!)^{k^2-2k}}  \left(\frac{(k-1)^{k(k-1)}}{k^{4(k-2)}}\right) \\ 
\ = \ & \frac{(k^3+k^2+k-1)!}{(k^2+k)(k!)^k((k+1)!)^{k^2-2k}} \cdot \frac{(k-1)^{k(k-1)}}{k^{4k-8}} \\ 
\ \leq \ & \frac{(k^3+k^2+k-1)!}{(k^2+k)(k!)^2} \left(\frac{(k-1)^{k(k-1)}}{k^{4k-8}k^{2k^2-4k}}\right) \leq \frac{(k^3+k^2+k-1)!}{(k^2+k)k^2} \\
\ = \ & \left(\frac{(k^3+k^2+k-1)}{k^2 (k^2+k)}\right) (k^3+k^2+k-2)! \\
\ < \ & \frac{2}{k} (k^3+k^2+k-2)! = (N_{k, 4}-3)!   
\end{align*}

Here we used that for any $\ell \geq 3$, $A_{\ell} < \ell!$. This proves the base case of the induction. 
    
    Now we prove the inductive step. Observe that by definition, $A_{\ell} < \ell!$. We derive that 
    \begin{equation*}\begin{split}
        & Z_{k, zig}(\ell+1)   \\ & = (k-1)^{k(k-1)}\binom{N_{k, \ell+1}-2}{\ell+1} \frac{A_{\ell+1} \cdot  (N_{k, \ell+1} -\ell -3)! \prod_{j=3}^{\ell} (T_{k, zig}(\ell))^{k-1} }{(N_{k, \ell}-1)!(N_{k, \ell}!)^{k-2}(N_{k, \ell-1}-1)!(N_{k, \ell-1}!)^{k-2} (N_{k, \ell-2})!^{k-1} (N_{k, \ell-3}!)^{k-1}\dots (N_{k, 1})!^{k-1}}  \\ & < \frac{(N_{k, \ell+1}-2)! (N_{k, 3}-2)! }{(N_{k, \ell}-1)!(N_{k, \ell}!)^{k-2}(N_{k, \ell-1}-1)!(N_{k, \ell-1}!)^{k-2} (N_{k, \ell-2}!)^{k-1} (N_{k, \ell-3}!)^{k-1}\dots (N_{k, 1}!)^{k-1}} \left(\prod_{j=4}^{\ell}(N_{k, j}-3)!\right)^{k-1}\end{split}.
    \end{equation*} Here, we use the inductive hypothesis in the last step along with the fact that $T_{k, zig}(3) = \beta_{k, 3} = A_3 \binom{N_{k, 3}-2}{3}\frac{1}{((N_{k, 2}-1)!)(N_{k, 2}!)^{k-2}(N_{k, 1}-1)! (N_{k, 1}!)^{k-2}}(k-1)^{k-1} < (N_{k, 3}-2)!$. 
    
    From there, we obtain \begin{equation*}\begin{split}
        & Z_{k, zig}(\ell+1)   \\&  < \frac{(N_{k, \ell+1}-2)! ((N_{k, 3}-2)!)^{k-1} (k-1)^{k-1} }{(N_{k, \ell}-1)!(N_{k, \ell}!)^{k-2}(N_{k, \ell-1}-1)!(N_{k, \ell-1}!)^{k-2} (N_{k, \ell-2}!)^{k-1} (N_{k, \ell-3}!)^{k-1}\dots (N_{k, 1}!)^{k-1}} \left(\prod_{j=4}^{\ell} (N_{k, j}-3)!\right)^{k-1} \\ & \leq \frac{(N_{k, \ell+1}-2)!}{(N_{k, \ell}-1)!(N_{k, \ell}!)^{k-2}} (k-1)^{k-1} \\ & \leq \frac{(N_{k, \ell+1}-2)!}{k^{k^{\ell-1}+k^2-k}}(k-1)^{k-1} \\ & \leq (N_{k, \ell+1}-3)!\end{split}.
    \end{equation*} with the penultimate inequality being a result of the fact that $\ell \geq 4$ and $k\geq 2$.
\end{proof}

\begin{example}
    Using Proposition~\ref{prop:Naive} and Theorem~\ref{thm:GeneralZigzagBound}, we compute upper bounds on $Z_{4, \ell}$ for $\ell = 3$, $4$, and $5$. They are displayed in Table~\ref{tab:Upperbdk4}.

    \begin{table}[H]
    \centering
    \begin{tabular}{|c|c|c|}
        \hline  $\ell$ & Naive Upper bound on $Z_{4, \ell}$ (Prop.~\ref{prop:Naive}) & The upper bound on $Z_{4, \ell}$ from Thm.~\ref{thm:GeneralZigzagBound} \\
         \hline 3 & 121645100408832000  & 3167841156480 \\ \hline 4 & $\approx 3.9\times 10^{124}$ & $ \approx 3.2  \times 10^{99}$  \\ \hline 5 & $\approx 1.5\times 10^{712}$ &  $\approx 2.0 \times 10^{601}$  \\ \hline
    \end{tabular}
    \caption{Upper bounds on $Z_{4, \ell}$ for $\ell=3, 4, 5$.}
    \label{tab:Upperbdk4}
\end{table}
\end{example}

\begin{remark*}
    Although the proofs of Theorem~\ref{thm:BinaryIKL} and Theorem~\ref{thm:GeneralZigzagBound} are similar, the bound on $Z_{2,\ell}$ obtained from the former theorem is not as tight as the one obtained from the latter theorem. To demonstrate this, for $k=2$, the case of binary trees, we compute Table~\ref{tab:ComparingBounds}, which displays the upper bounds on the number of configurations obtained from those Theorems.
    \begin{table}[H]
        \centering
        \begin{tabular}{|c|c|c|}\hline
             $\ell$  &  Upper Bound on $Z_{2, \ell}$ From Thm.~\ref{thm:BinaryIKL}& Upper Bound on $Z_{2, \ell}$ From Thm.~\ref{thm:GeneralZigzagBound} \\ \hline 4   & 693,000 & 18,018,000  \\ \hline 5
             & $\approx 2.9\times 10^{22}$  & $\approx 1.1\times10^{24}$ \\ \hline 6  & $\approx 1.8\times 10^{65}$ & $\approx 2.5 \times 10^{67}$ \\ \hline  7& $\approx 1.5 \times 10^{170}$   & $\approx 3.1 \times 10^{173}$  \\ \hline
        \end{tabular}
        \caption{Upper bounds on $Z_{2, \ell}$ from Theorems~\ref{thm:GeneralZigzagBound} and \ref{thm:BinaryIKL} for $\ell=4, 5, 6, 7$.}
        \label{tab:ComparingBounds}
    \end{table}
\end{remark*}
\section{Lower Bounds on the Number of Stable Configurations}\label{sec:LowerBounds}
Finally, we provide lower bounds for $Z_{k, \ell}$ through direct constructions.

For this, we introduce some notations. Given $N_{k, \ell}$ chips at the root of a $k$-ary tree, define a \textit{stationary chip} as a designated chip that does not leave the root during the stabilization. We define a \textit{uneasy chip} to be any chip $c$ that 1) goes to one of the left children of the root and is larger than the stationary chip or 2) goes to one of the right children of the root and is smaller than the stationary chip. We say that a firing procedure fires \textit{symmetrically} if every firing of the root involves the stationary chip, the stationary chip remains at the root after each firing, and every chip fired to the root is sent back to the child of the root from which it came.

\begin{figure}[h]
    \centering
    \includegraphics[width=0.65\linewidth]{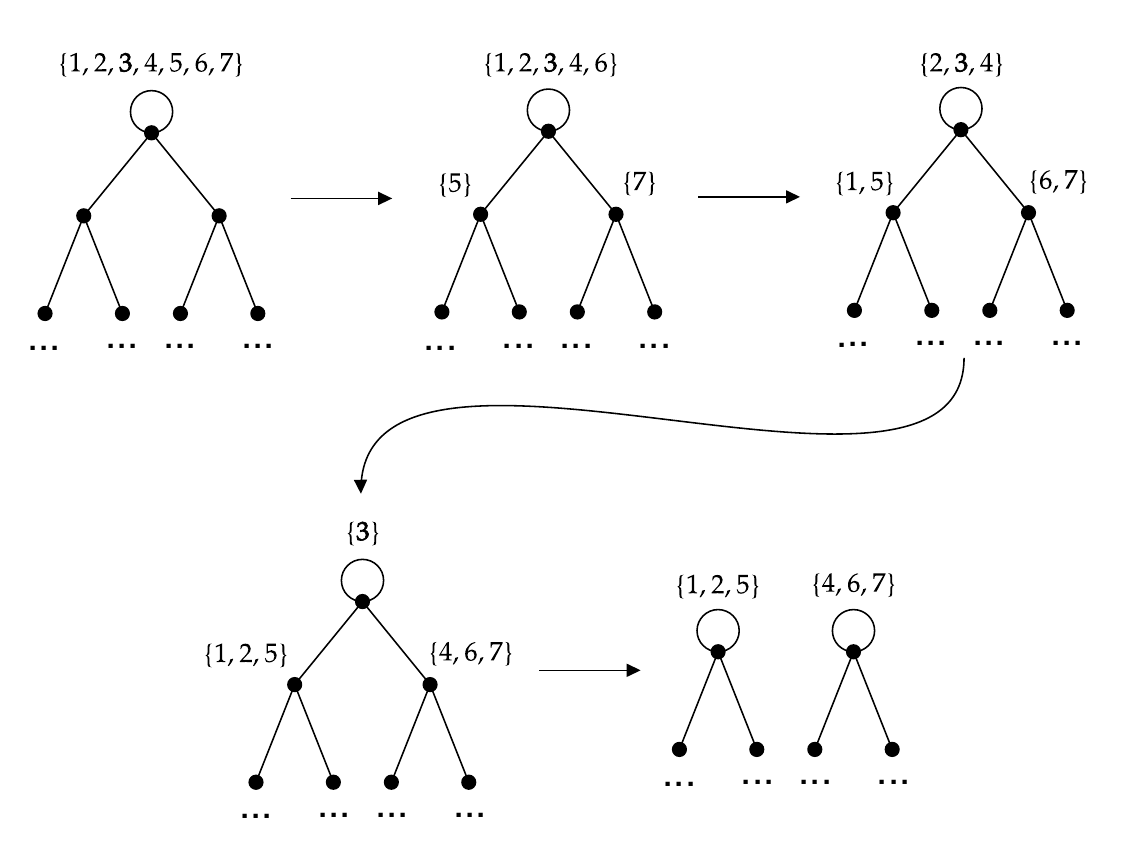}
    \caption{A visualization of $(Z_{2, 3})^2$ possible stable configurations with $m-1 = 3$ as the stationary chip and 5 as the uneasy chip on a binary tree with three layers in the stable configuration.}
    \label{fig:lower_bound_construction}
\end{figure}

\begin{example}
    Consider starting with chips $1,2, \dots, 7$ in an undirected binary tree as in Figure~\ref{fig:lower_bound_construction}. As in the figure, we fire first the triple $\{5, 6, 7\}$, then $\{1, 3, 6\}$, and afterwards $\{2, 3, 4\}$ from the root. Then we fire $1,2,5$ from the left child of the root, and we fire $4, 6, 7$ from the right child of the root. Finally, we fire $2,3, 6$ from the root. 

    In this stabilization, we say that the chip $3$ is the stationary chip, since it stays at the root vertex. We say that chip $5$ is an uneasy chip, since it is bigger than the stationary chip, chip 3, and is sent to the subtree rooted at the left child of the root.
\end{example}

\begin{lemma}\label{lem:lowerbound}
For positive integers $k \geq 2$ and $\ell \geq 3$, we have the following lower bound for $Z_{k,\ell}$:$$Z_{k, \ell} \geq \sum_{i=0}^{\lfloor \frac{k}{2} \rfloor}\binom{(N_{k, \ell}-1) \cdot \frac{\lfloor \frac{k}{2} \rfloor}{k} -1-\lfloor \frac{k}{2}  \rfloor}{i} \binom{(N_{k, \ell}-1) \cdot \frac{\lceil \frac{k}{2} \rceil}{k} -1-2\lceil \frac{k}{2}  \rceil+i}{i}Z_{k, \ell-1}^{k}$$
\end{lemma}

\begin{proof}

Consider the chip with label $m = (N_{k, \ell}-1)\frac{\lfloor \frac{k}{2} \rfloor}{k}+1$. We provide a construction where $m$ is the stationary chip after two initial fires of the root vertex.

We now propose the following firing procedure. Fix $i \in [0, \lfloor \frac k 2 \rfloor ]$. Then,
\begin{enumerate}
    \item Pick $i$ chips $c_1, c_2, \dots, c_i$ among $\lfloor \frac{k}{2} \rfloor+2, \lfloor \frac{k}{2} \rfloor+3, \dots, (N_{k, \ell}-1)\frac{\lfloor \frac{k}{2} \rfloor}{k}$ such that $c_1 < c_2 < \dots < c_i$.
    \item From the root, fire the tuple of $k+1$ chips $(1, 2, \dots, \lfloor \frac{k}{2} \rfloor,\lfloor \frac{k}{2} \rfloor+1, c_1, c_2, \dots, c_i, N_{k, \ell} -\lceil \frac{k}{2} \rceil +i +1, N_{k, \ell} -\lceil \frac{k}{2} \rceil +i +2,  \dots, N_{k, \ell})$. For each $j \in \{1, 2, \dots, i\}$, this sends $c_i$ to the $\lfloor \frac{k}{2} \rfloor+i$th leftmost child of the root as an uneasy chip.
    \item Pick $i$ chips $c_1', c_2', \dots, c_i'$ among $(N_{k, \ell}-1)\frac{\lfloor \frac{k}{2} \rfloor}{k}+2, (N_{k, \ell}-1)\frac{\lfloor \frac{k}{2} \rfloor}{k}+3, \dots, N_{k, \ell} - 2\lceil\frac{k}{2} \rceil+i-1$ so that $c_1' < c_2' < \dots < c_i'$.) 
    \item Consider the $\lfloor\frac{k}{2} \rfloor -i$ smallest chips among $\{1, 2, \dots, m-1\} \setminus \{1,2, \dots, \lfloor  \frac{k}{2}\rfloor, c_1, c_2, \dots, c_i\}$. Denote them as $d_1, d_2, \dots, d_{\lfloor\frac{k}{2} \rfloor -i}$ and assume without loss of generality that they are listed in increasing order. Fire tuple of chips $(d_1, d_2, \dots, d_{\lfloor \frac{k}{2} \rfloor - i}, c'_1, c'_2, \dots, c'_i, N_{k, \ell} - 2\lceil \frac{k}{2} \rceil+i, N_{k, \ell} - 2\lceil \frac{k}{2} \rceil +i+1, \dots, N_{k, \ell} - \lceil \frac{k}{2} \rceil+i)$ from the root. For each $j \in \{1, 2, \dots, i\}$, this introduces $c_j'$ to the $\lfloor \frac{k}{2} \rfloor - i + j$th leftmost subtree as an uneasy chip.
    \item  Observe that at this stage there are exactly $(N_{k, \ell}-1)\frac{\lfloor \frac{k}{2} \rfloor}{k} - (\lfloor \frac{k}{2} \rfloor+i +\lfloor \frac{k}{2}\rfloor - i)$ chips less than $m$ on the root, and there are exactly $(N_{k, \ell}-1)\frac{\lceil \frac{k}{2} \rceil}{k} - (\lceil \frac{k}{2} \rceil-i +\lceil \frac{k}{2}\rceil + i)$ chips greater than $m$ on the root vertex. Thus, the ratio of the number of chips less than $m$ over the number of chips greater than $m$ is exactly $\frac{\lfloor \frac{k}{2} \rfloor}{\lceil \frac{k}{2} \rceil}$, we can symmetrically fire. Fire the root symmetrically until we have $N_{k, \ell-1}$ chips on each child of the root. 
    \item Symmetrically fire until we reach a stable configuration. 
\end{enumerate}

Since we fire symmetrically until we reach a stable configuration, we can treat the chip-firing game in Step 6 as $k$ separate $k$-ary trees with self-loop and the $N_{k, \ell-1}$ labeled chips.

Now, for each $i \in \{0, 1, \dots, \lfloor \frac k 2 \rfloor\}$, we find a lower bound on the number of chip configurations that can result from the above procedure. First, observe that there are $\binom{(N_{k, \ell}-1) \cdot \frac{\lfloor \frac{k}{2} \rfloor}{k} -1-\lfloor \frac{k}{2}  \rfloor}{i}$ ways to choose $c_1, c_2, \dots, c_i$ as in step 1 of the above procedure. Then, observe that there are $\binom{(N_{k, \ell}-1) \cdot \frac{\lceil \frac{k}{2} \rceil}{k} -1-2\lceil \frac{k}{2}  \rceil+i}{i}$ ways to choose $c_1', c_2', \dots, c_i'$ as in step 3 of the above procedure. From symmetrically firing the root, we find that at least one stable configuration can be obtained such that for each $j \in \{1, 2, \dots, i\}$
\begin{itemize}
 \item Chip $c_j$, as an uneasy chip on the $\lfloor \frac{k}{2} \rfloor + j$th leftmost child of the root, is at the subtree rooted at that child of the root. This is because all other chips that are fired to the $\lfloor \frac{k}{2} \rfloor + j$th leftmost child of the root (or at the subtree rooted at the child) are greater than $m$ whereas $c_j < m$.
    \item Chip $c_j'$, as an uneasy chip on the $\lfloor \frac{k}{2} \rfloor - i + j$th leftmost child of the root, is at the subtree rooted at that child of the root. This is because all other chips that are fired to the $\lfloor \frac{k}{2} \rfloor -i+ j$th leftmost child (or at the subtree rooted at the child) of the root are less than than $m$ whereas $c_j' > m$.
\end{itemize} We therefore obtain that the number of stable configurations resulting from the procedure above is at least $(\binom{(N_{k, \ell}-1) \cdot \frac{\lfloor \frac{k}{2} \rfloor}{k} -1-\lfloor \frac{k}{2}  \rfloor}{i} \binom{(N_{k, \ell}-1) \cdot \frac{\lceil \frac{k}{2} \rceil}{k} -1-2\lceil \frac{k}{2}  \rceil+i}{i})Z_{k, \ell-1}^{k}$. Now, summing this over all $i$, we obtain that the number of possible stable configurations resulting from the procedure above is $$\sum_{i=0}^{\lfloor \frac{k}{2} \rfloor}\binom{(N_{k, \ell}-1) \cdot \frac{\lfloor \frac{k}{2} \rfloor}{k} -1-\lfloor \frac{k}{2}  \rfloor}{i} \binom{(N_{k, \ell}-1) \cdot \frac{\lceil \frac{k}{2} \rceil}{k} -1-2\lceil \frac{k}{2}  \rceil+i}{i}Z_{k, \ell-1}^{k}$$ Therefore, we find that $Z_{k,\ell}$, number of stable configurations that can result from starting with $N_{k, \ell}$ labeled chips on the $k$-ary tree, is at least $$\sum_{i=0}^{\lfloor \frac{k}{2} \rfloor}\binom{(N_{k, \ell}-1) \cdot \frac{\lfloor \frac{k}{2} \rfloor}{k} -1-\lfloor \frac{k}{2}  \rfloor}{i} \binom{(N_{k, \ell}-1) \cdot \frac{\lceil \frac{k}{2} \rceil}{k} -1-2\lceil \frac{k}{2}  \rceil+i}{i}Z_{k, \ell-1}^{k}.$$
\end{proof}

\begin{example}
    In Figure~\ref{fig:LemmaConstruction}, we illustrate the construction in the proof of Lemma~\ref{lem:lowerbound} on the binary tree with $N_{2, 4}=15$ labeled chips. We first pick uneasy chip $c_1 = 3$, then fire tuple of chips $(1, 2, 3)$ from the root. Then we pick $c_1'=13$. We fire $(13, 14, 15)$ from the root. Firing symmetrically, we obtain a configuration where $8$ is at the root, chips $1,2,4,5,6,7,13$ are on the left child of the root, and chips $3,9,10,11,12,14, 15$ are on the right child of the root. We keep firing symmetrically to obtain the stable configuration at the end of the flowchart.

    \begin{figure}[H]
    \centering
    \includegraphics[width=0.8\linewidth]{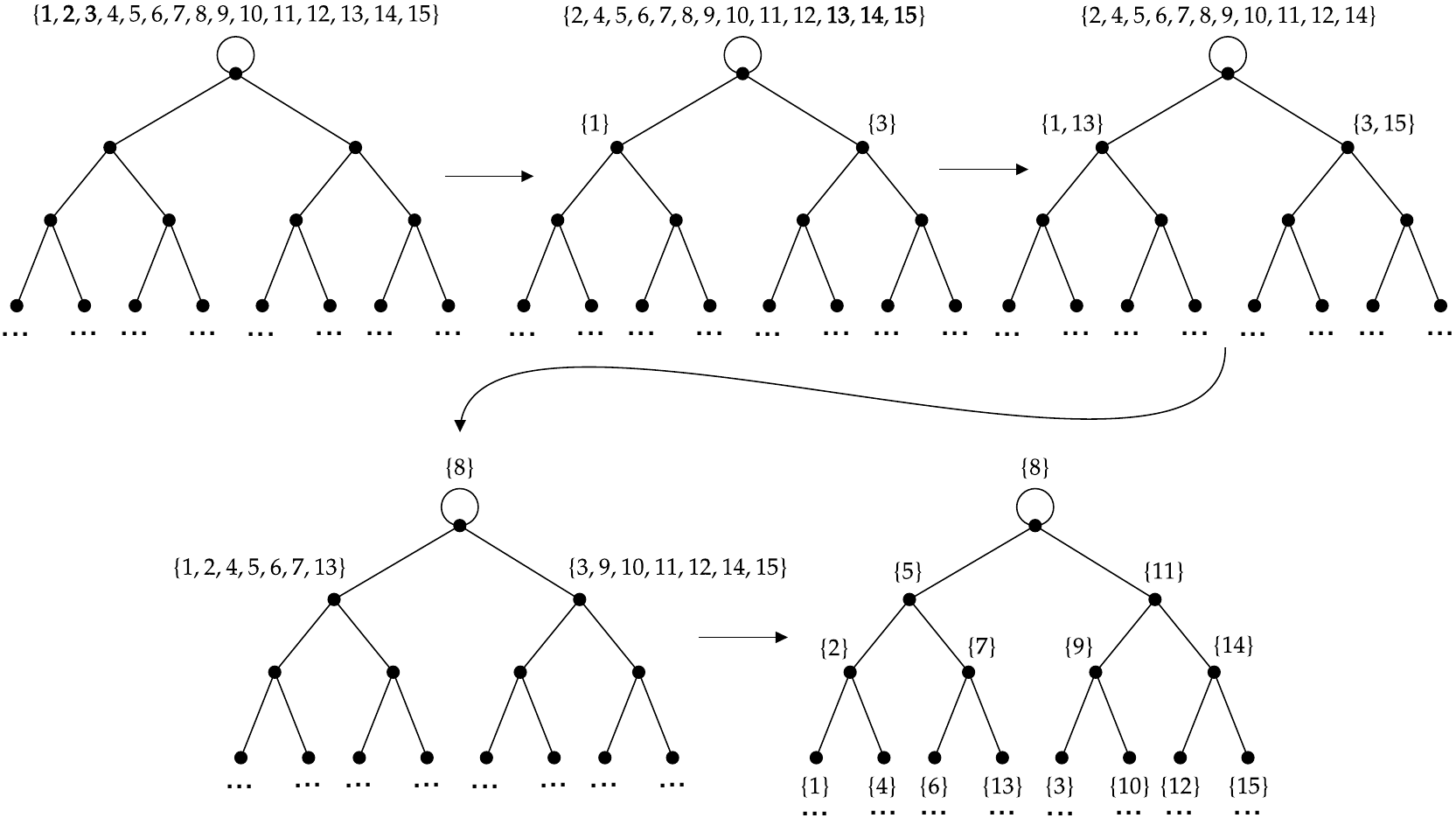}
    \caption{The construction from the proof of Lemma~\ref{lem:lowerbound} applied on binary tree with $N_{2, 4}$ labeled chips.}
    \label{fig:LemmaConstruction}
\end{figure}
\end{example}

We now use this recursive bound to find a closed-form expression that bounds $Z_{k, \ell}$ from below.

\begin{theorem}\label{thm:LowerBoundClosedForm2} For $\ell \geq 3$,
    \[Z_{2, \ell} \geq 6^{2^{\ell-3}}\prod_{j=4}^{\ell}\left(1 + \left(\frac{N_{2, j}-1}{2}-2\right)^2\right)^{2^{\ell-j}}.\]
\end{theorem}
\begin{proof}
We prove the above statement via proof by induction over $\ell$.

We first establish the base case of theorem. In Section 7 of \cite{MR4827886}, Musiker and Nguyen find that $Z_{2, 3} = 6$. Therefore $Z_{2, 3} \geq 6^{2^{3-3}}\prod_{j=4}^{3}\left(1+\left(\frac{N_{2, j}-1}{2}-2\right)^2\right)^{2^{\ell-j}}=6$.

Next, for the inductive step, assume that for $\ell \geq 3$, $$Z_{2, \ell} \geq 6^{2^{\ell-3}}\prod_{j=4}^{\ell}\left(1 + \left(\frac{N_{2, j}-1}{2}-2\right)^2\right)^{2^{\ell-j}}.$$

From, Lemma~\ref{lem:lowerbound}, we find that \begin{equation*}
    \begin{split}
        Z_{2, \ell}\geq \sum_{i=0}^{1}\binom{(N_{k, \ell}-1) \cdot \frac{1}{2} -2}{i} \binom{(N_{2, \ell+1}-1) \cdot \frac{1}{2} -1-2+i}{i}Z_{2, \ell}^{2} \\ \geq \left(1 + \left(\frac{N_{2, \ell+1}-1}{2}-2\right)^2\right)\left(6^{2^{\ell-3}}\prod_{j=4}^{\ell}\left(1 + \left(\frac{N_{2, j}-1}{2}-2\right)^2\right)^{2^{\ell-j}} \right)^2 \\ \geq 6^{2^{\ell-2}}\prod_{j=4}^{\ell+1}\left(1 + \left(\frac{N_{2, j}-1}{2}-2\right)^2\right)^{2^{\ell+1-j}}.
    \end{split}
\end{equation*} This completes the proof.
\end{proof}

\begin{example}
    In Table~\ref{tab:ComputationalLowerk2}, we calculate lower bounds for $Z_{2, 4}, Z_{2, 5}, Z_{2, 6}$ and $Z_{2, 7}$ using Theorem~\ref{thm:LowerBoundClosedForm2} and compare them to the known values and upper bounds of $Z_{2, \ell}$. Note that, in this table, the lower bound on $Z_{2, \ell}$ is orders of magnitude smaller than the upper bounds for $Z_{2, \ell}$ listed in Table \ref{tab:ComparingBounds}.
\begin{table}[H]
    \centering
    \begin{tabular}{|c|c|c|c|}
         \hline $\ell$ &  Lower Bound on $Z_{2, \ell}$  & Upper Bound on $Z_{2, \ell}$ From Thm. ~\ref{thm:BinaryIKL} & Upper Bound on $Z_{2, \ell}$ From Thm.~\ref{thm:GeneralZigzagBound} \\ \hline
         4 & $936$ & 693,000 & 18,018,000 \\ \hline
         5 & 148936320   & $\approx 2.9\times 10^{22}$ & $\approx 1.1 \times 10^{24}$ \\ \hline
         6 & $\approx 1.9 \times 10^{19}$ & $\approx 1.8\times 10^{65}$ & $\approx 2.5 \times 10^{67}$ \\ \hline 7 & $\approx 1.3 \times 10^{42}$ & $\approx 1.5 \times 10^{170}$ & $\approx 3.1 \times 10^{173}$ \\ \hline    \end{tabular}
    \caption{Lower and Upper bounds for $Z_{2, \ell}$ computed for $\ell=4, 5, 6, 7$}
    \label{tab:ComputationalLowerk2}
\end{table}
\end{example}

For general $k \geq 2$, we have the following lower bound ond on 
\begin{theorem}\label{thm:LowerBoundClosedFormk}

For $\ell \ge 3$ and $k \geq 2$, we have: \[Z_{k,\ell} \ge \prod_{j=3}^{\ell} \left(\sum_{i=0}^{\lfloor \frac{k}{2} \rfloor}\binom{(N_{k, \ell}-1) \cdot \frac{\lfloor \frac{k}{2} \rfloor}{k} -1-\lfloor \frac{k}{2}  \rfloor}{i} \binom{(N_{k, \ell}-1) \cdot \frac{\lceil \frac{k}{2} \rceil}{k} -1-2\lceil \frac{k}{2}  \rceil+i}{i} \right)^{k^{\ell-j}}.\]

\end{theorem}

\begin{proof}

Let $\eta_{k, \ell}$ be shorthand for $$\eta_{k, \ell} =\sum_{i=0}^{\lfloor \frac{k}{2} \rfloor}\binom{(N_{k, \ell}-1) \cdot \frac{\lfloor \frac{k}{2} \rfloor}{k} -1-\lfloor \frac{k}{2}  \rfloor}{i} \binom{(N_{k, \ell}-1) \cdot \frac{\lceil \frac{k}{2} \rceil}{k} -1-2\lceil \frac{k}{2}  \rceil+i}{i}.$$

We prove by induction. First, notice that by Lemma~\ref{lem:lowerbound} and Lemma~\ref{lem:small}, we have \begin{align*}
Z_{k, 3} & \ \ge \ \eta_{k, 3} Z^k_{k, 2} = \prod_{j=3}^{3} \eta_{k, j} \\ 
& \ = \prod_{j=3}^{3} \left(\sum_{i=0}^{\lfloor \frac{k}{2} \rfloor}\binom{(N_{k, j}-1) \cdot \frac{\lfloor \frac{k}{2} \rfloor}{k} -1-\lfloor \frac{k}{2}  \rfloor}{i} \binom{(N_{k, j}-1) \cdot \frac{\lceil \frac{k}{2} \rceil}{k} -1-2\lceil \frac{k}{2}  \rceil+i}{i}\right)^{k^{3-j}}.
\end{align*}

Next, for the inductive step, assume that $$Z_{k,\ell} \ge \prod_{j=3}^{\ell} \eta_{k, j}^{k^{\ell-j}}.$$

By Lemma~\ref{lem:lowerbound}, we know that 
\begin{align*}
Z_{k, \ell+1} & \ge \eta_{k, \ell+1} Z^k_{k, \ell}  \ge \eta_{k, \ell+1}\left(\prod_{j=3}^{\ell}\eta_{k, j}^{k^{\ell-j}}\right)^k  = \prod_{j=3}^{\ell+1}\eta_{k, j}^{k^{\ell+1-j}}
\end{align*}

This proves the inductive step of the proof.
\end{proof}

\begin{example}
    For $k=2$, we find that $\eta_{2, 3}$, as defined in the proof of Theorem~\ref{thm:LowerBoundClosedFormk}, is $2$. Therefore, the Theorem yields lower bound $Z_{2, \ell} \geq 2^{\ell-3}\prod_{j=4}^{\ell}\left(1 + (\frac{N_{2, j}-1}{2}-2)^2\right)^{k^{\ell-j}}$. Thus, the bound for $Z_{2, \ell}$ from Theorem~\ref{thm:LowerBoundClosedFormk} is weaker than the bound from Theorem~\ref{thm:LowerBoundClosedForm2}.
\end{example}

\begin{example}
    Table~\ref{tab:ComputationalLowerk4} presents lower and upper bounds for $Z_{4, 3}, Z_{4, 4}, Z_{4, 5}$. We computed the lower bounds using Theorem~\ref{thm:LowerBoundClosedFormk} and upper bounds using Theorem~\ref{thm:GeneralZigzagBound}.
    
    Note that the lower bounds for $Z_{4, 3}, Z_{4, 4}, Z_{4, 5}$ are many orders of magnitude smaller than their respective upper bounds. Thus, a natural open question is whether there are tighter upper and lower bounds on the number of stable configurations $Z_{k, \ell}$ of the $k$-ary tree with $N_{k, \ell}$ labeled chips.
\begin{table}[H]
    \centering
    \begin{tabular}{|c|c|c|}
         \hline $\ell$ &  Lower bound on $Z_{4, \ell}$ & Upper Bound on $Z_{4, \ell}$ From Thm.~\ref{thm:GeneralZigzagBound}\\ \hline
         $3$ & 484 &  3167841156480 \\\hline $4$ & $\approx 3.02\times 10^{16}$  & $\approx 3.2146  \times 10^{99}$ \\ \hline $5$ & $\approx 1.6 \times 10^{74}$ & $\approx 1.9761 \times 10^{601}$\\ \hline
    \end{tabular}
    \caption{Lower and upper bounds for $Z_{4, \ell}$ computed for $\ell=3, 4, 5$}
    \label{tab:ComputationalLowerk4}
\end{table}
\end{example}

\section{Further Directions}\label{sec:Further}
In this section, we discuss conjectures and further questions related to the properties and numbers of stable configurations of labeled chips in the $k$-ary tree.
\subsection{Ballot Property}

In \cite{inagaki2024chipfiringundirectedbinarytrees}, the first author, Khovanova, and Luo conjectured that stable configurations of $2^{\ell}-1$ labeled chips on a binary tree exhibit the \emph{ballot property}: for every vertex \(v\) and integer \(i\), the $i$th smallest chip to the left of \(v\) is less than the $i$th smallest chip to the right of $v$. This resembles ballot sequences counted by Catalan numbers \cite{MR3467982}. Assuming this property, the authors established a bound on the number of stable configurations. They confirmed it for $\ell=3$ and $\ell=4$, but the validity for higher $\ell$ remains uncertain.

In light of this, we conjecture the following generalized ballot property:
\begin{conjecture}
    Consider any stable configuration resulting from $N_{k, \ell}$ labeled chips starting at the root of the $k$-ary tree. Then for each vertex $v$, $i\in \{1,2, \dots, N_{k, \ell-1}\}$ and $a, b \in \{1,2, \dots, k\}$ such that $a < b$, the $i$th smallest chip that is in the subtree rooted at the $a$th leftmost child of $v$ is less than the $i$th smallest chip in the subtree rooted at the $b$th leftmost child of $v$.
\end{conjecture}

We believe that if this conjecture is true, we can find another upper bound on the number of configurations of $N_{k, \ell}$ labeled chips on the $k$-ary tree; it would be proven similarly to Theorem 3.19 of \cite{inagaki2024chipfiringundirectedbinarytrees}.

\subsection{The Stable Configurations as a Permutation of $1, 2, \dots, N_{k, \ell}$}
As noted in Musiker and Nguyen, one can view the stable configuration of $N_{k, \ell}$ chips at the root as a stable configuration. More specifically, they considered listing all of the chips in the stable configuration of the binary tree from left to right.  Now, consider doing this for $k$-ary trees; for odd $k$ and for each vertex $v$ of the tree, we say that the $\lceil \frac{k}{2}\rceil$nd leftmost child of $v$ is to the right of the $v$ and if vertex $v$ is in a subtree whose root is left of that of a subtree containing $u$ then $v$ is left of $u$. This is better explained via illustration. We call this permutation the flattened configuration.

\begin{example}
    Consider the stable configuration of chips in Figure~\ref{fig:endgameexample}. We read it as the permutation $1,2, 3, 4, 9, 5, 6, 7, 8, 10, 12, 14, 15, 16, 17, 13, 18, 19, 20, 21$ of $1,2, \dots, 21.$
\end{example}

For $k=2$ and $n=4$, Patrick Liscio \cite{liscio} lists out all configurations as permutations derived from the flattened configuration with the number of inversions, i.e., the number of pairs of indices $i, j$ such that both $i < j$ and the $i$th term of the flattened configuration is greater than the $j$th. Viewing the configurations as permutations, we ask the following questions:
\begin{itemize}
    \item What is the maximum number of inversions that can appear in a flattened configuration of $N_{k, \ell}$ chips?
    \item What permutations and patterns of the flattened configuration are avoided or can show up?
\end{itemize}
In \cite{MR4887467}, these questions were addressed in the context of labeled chip firing on \emph{directed} $k$-ary trees. 

From Liscio's data \cite{liscio}, 25 is the largest number of inversions that appear in a stable configuration of $2^4-1$ labeled chips on a binary tree. For $n \geq 5$, the largest number of inversions in a stable configuration of \(2^{\ell}-1\) remains unknown due to computational difficulties of the task \cite{MR4827886}.

To address the first question, one can almost immediately observe that the maximum number of inversions that can appear is less than $\frac{N_{k, \ell}(N_{k, \ell}-1)}{2}$. This is because the permutation $N_{k, \ell}, N_{k, \ell}-1, \dots, 1$ cannot be a flattened configuration of any stable configuration; indeed, chip $1$ must be the straight left bottom descendant of the root and chip $N_{k, \ell}$ must be the straight right bottom descendant of the root by Proposition~\ref{prop:smallestchip}. 

To address the second question, we believe that relations like Theorem~\ref{zigzagrelation}, which show relative positions of chips based on numerical order, can reveal patterns. Given the limited literature on viewing undirected $k$-ary trees as permutations, Musiker and Nguyen's results (e.g., Proposition~\ref{4.4}) may be a good starting point for exploration. 

\section{Acknowledgements}

We thank Tanya Khovanova for her numerous helpful comments, insightful conversations, and proofreading.  We would like to extend our gratitude to Prof. Jonathan Bloom, Prof. Roman Bezrukavnikov, and AnaMaria Perez, who oversaw the progress of the research problem insightful comments. We also thank Professor Alexander Postnikov for useful discussions.
 
This project was started in the 2025 Research Science Institute (RSI) program, run by the Center for Excellence in Education (CEE) and hosted by the Massachusetts Institute of Technology (MIT). The MIT Department of Mathematics supports the first author. Citadel Securities sponsored the second author during RSI.

Figures were created with the help of TikZ and Mathcha.
\bibliographystyle{plain}

\begin{thebibliography}{10}

\bibitem{agrawal2025chipfiringinfinitekarytrees}
Dillan Agrawal, Selena Ge, Jate Greene, Tanya Khovanova, Dohun Kim, Rajarshi Mandal, Tanish Parida, Anirudh Pulugurtha, Gordon Redwine, Soham Samanta, and Albert Xu.
\newblock Chip-firing on infinite $k$-ary trees, 2025.

\bibitem{anderson1989disks}
Richard Anderson, L{\'a}szl{\'o} Lov{\'a}sz, Peter Shor, Joel Spencer, {\'E}va Tardos, and Shmuel Winograd.
\newblock Disks, balls, and walls: analysis of a combinatorial game.
\newblock {\em The American Mathematical Monthly}, 96(6):481--493, 1989.

\bibitem{bak1987self}
Per Bak, Chao Tang, and Kurt Wiesenfeld.
\newblock Self-organized criticality: An explanation of the 1/f noise.
\newblock {\em Physical Review Letters}, 59(4):381, 1987.

\bibitem{MR3144399}
Andrew Berget.
\newblock Critical groups of graphs with reflective symmetry.
\newblock {\em J. Algebraic Combin.}, 39(1):209--224, 2014.

\bibitem{bjorner1991chip}
Anders Bj{\"o}rner, L{\'a}szl{\'o} Lov{\'a}sz, and Peter~W Shor.
\newblock Chip-firing games on graphs.
\newblock {\em European Journal of Combinatorics}, 12(4):283--291, 1991.

\bibitem{dhar1990self}
Deepak Dhar.
\newblock Self-organized critical state of sandpile automaton models.
\newblock {\em Physical Review Letters}, 64(14):1613, 1990.

\bibitem{dhar1999abelian}
Deepak Dhar.
\newblock The {A}belian sandpile and related models.
\newblock {\em Physica A: Statistical Mechanics and its applications}, 263(1-4):4--25, 1999.

\bibitem{zbMATH06585703}
Johnny Guzm{\'a}n and Caroline Klivans.
\newblock Chip firing on general invertible matrices.
\newblock {\em SIAM J. Discrete Math.}, 30(2):1115--1127, 2016.

\bibitem{Hopkins_2017}
Sam Hopkins, Thomas McConville, and James Propp.
\newblock Sorting via chip-firing.
\newblock {\em The Electronic Journal of Combinatorics}, 24(3), 2017.

\bibitem{MR4887467}
Ryota Inagaki, Tanya Khovanova, and Austin Luo.
\newblock Chip firing on directed {$k$}-ary trees.
\newblock {\em Enumer. Comb. Appl.}, 5(2):Paper No. S2R16, 14, 2025.

\bibitem{inagaki2024chipfiringundirectedbinarytrees}
Ryota Inagaki, Tanya Khovanova, and Austin Luo.
\newblock On chip-firing on undirected binary trees.
\newblock {\em Annals of Combinatorics}, Aug 2025.

\bibitem{klivans2018mathematics}
Caroline~J Klivans.
\newblock {\em The Mathematics of Chip-firing}.
\newblock Chapman and Hall/CRC, 2018.

\bibitem{liscio}
Patrick Liscio.
\newblock Liscio’s result.txt.
\newblock Available at \url{{https://drive.google.com/file/d/1b_QkfDahUJuDoF7YhMf_oCU6qNcSY_Sj/view}}.

\bibitem{MR4827886}
Gregg Musiker and Son Nguyen.
\newblock Labeled chip-firing on binary trees with {$2^n-1$} chips.
\newblock {\em Ann. Comb.}, 28(4):1167--1197, 2024.

\bibitem{spencer1986balancing}
Joel Spencer.
\newblock Balancing vectors in the max norm.
\newblock {\em Combinatorica}, 6(1):55--65, 1986.

\bibitem{MR3467982}
Richard~P. Stanley.
\newblock {\em Catalan numbers}.
\newblock Cambridge University Press, New York, 2015.

\end{thebibliography}

\smallskip

\noindent
Ryota Inagaki \\
\textsc{
Department of Mathematics, Massachusetts Institute of Technology\\
77 Massachusetts Avenue, Building 2, Cambridge, Massachusetts, U.S.A. 02139}\\
\textit{E-mail address: }\texttt{inaga270@mit.edu}
\medskip

\noindent
Aaron Lin \\
\textsc{
Ladue Horton Watkins High School,\\
1201 South Warson Road, St. Louis, MO 63124}\\
\textit{E-mail address: }\texttt{aaronlin0924@gmail.com}
\medskip

\end{document}